\documentclass[a4]{amsart}
\usepackage{amsfonts,amsmath,amscd,amssymb,amsthm}
\usepackage{enumerate,hyperref,url}

\theoremstyle{plain}
\newtheorem{theorem}{Theorem}

\newtheorem{corollary}[theorem]{Corollary}

\newtheorem{lemma}[theorem]{Lemma}
\newtheorem{proposition}[theorem]{Proposition}

\theoremstyle{remark}
\newtheorem*{example}{\textbf{Example}}
\newtheorem{remark}[theorem]{Remark}

\numberwithin{equation}{section}
\allowdisplaybreaks 

\newcommand\R{\mathbb R}
\newcommand\RR{\mathbb R}
\newcommand\ZZ{\mathbb Z}
\newcommand{\xx}{\mathbf x}

\newcommand{\nn}{\mathbf n}

\newcommand{\pO}{{\partial\Omega}}
\newcommand{\pM}{{\partial M}}

\newcommand\II{\text{II} }

\newcommand{\vol}{\mathrm{vol}}
\newcommand\Id{\mathrm{Id}}
\newcommand\etabar{\overline{\eta}}
\newcommand\taubar{\overline{\tau}}
\newcommand{\scal}{\mathcal{S}}

\title[Completeness of boundary traces of eigenfunctions]{Completeness of boundary traces of eigenfunctions}

\author{Xiaolong Han}
\address{Mathematical Sciences Institute, Australian National University, Canberra, ACT 2601, Australia}
\email{Xiaolong.Han@anu.edu.au}

\author{Andrew Hassell}
\address{Mathematical Sciences Institute, Australian National University, Canberra, ACT 2601, Australia}
\email{Andrew.Hassell@anu.edu.au}

\author{Hamid Hezari}
\address{Department of Mathematics, University of California, Irvine, CA 92697, USA}
\email{hezari@math.uci.edu}

\author{Steve Zelditch}
\address{Department of Mathematics, Northwestern University, Evanston, IL 60208, USA}
\email{zelditch@math.northwestern.edu}



\begin{document}
\maketitle

\begin{abstract}
In this paper, we study the boundary traces of eigenfunctions on the boundary of a smooth and bounded domain. An identity derived by B\"acker, F\"urstburger, Schubert, and Steiner \cite{BFSS}, expressing (in some sense) the asymptotic completeness of the set of boundary traces in a frequency window of size $O(1)$, is proved both for Dirichlet and Neumann boundary conditions. 
We then prove a semiclassical generalization of this identity. 
\end{abstract}

\section{Introduction and main results}

Given a compact and smooth $n$-dimensional Riemannian manifold $(M,g)$ with smooth boundary, let $\{u_j\}_{j=1}^{\infty}$ be an orthonormal basis of eigenfunctions of the positive Dirichlet Laplacian $\Delta_D$ with eigenvalues
$$0<\lambda_1^2 < \lambda_2^2 \le \lambda_3^2 \le\cdots,$$
that is
$$\begin{cases}
\Delta u_j=\lambda_j^2u_j & \text{ in }M,\\
u_j=0 & \text{ on }\pM,\\
\langle u_j,u_k\rangle_{M}=\delta_{jk}.
\end{cases}$$
in which 
$$\Delta=\Delta_g=-\frac{1}{\sqrt g}\sum_{i,j=1}^n\frac{\partial}{\partial x_i}\left(g^{ij}\sqrt g\frac{\partial}{\partial x_j}\right)$$ 
is the (positive) Laplacian and $\langle\cdot,\cdot\rangle_{M}$ denotes the inner product in $L^2(M)$. We define $\psi_j$ to be the exterior normal derivative of $u_j$ on the boundary: 
$$\psi_j:=d_nu_j|_{\pM}.$$

We also study the Neumann Laplacian $\Delta_N$. In this case, we denote an orthonormal basis of eigenfunctions by $\{v_j\}_{j=1}^\infty$ and $\{\omega_j\}_{j=1}^\infty$ their boundary values. 

This paper is motivated by the following beautiful identity of B\"acker, F\"urstburger, Schubert, and Steiner \cite[Equations 53--55]{BFSS}, expressing a sort of asymptotic completeness property of the boundary traces $\psi_j$ of Dirichlet eigenfunctions. 

\begin{theorem}[Completeness of boundary traces of Dirichlet eigenfunctions]\label{thm:BFSS}
Let $\rho\in\mathcal S(\R)$ be such that $\hat\rho$ is identically $1$ near $0$, and has sufficiently small support.  Then for any $\phi\in C^\infty(\pM)$, we have 
\begin{equation}
\phi(x)=\lim_{\lambda\to\infty} \frac{\pi}{2}\sum_j\rho(\lambda-\lambda_j)\lambda_j^{-2}\langle\psi_j,\phi\rangle\psi_j(x),
\label{eq:BFSS}\end{equation}
where $\langle\cdot,\cdot\rangle=\langle\cdot,\cdot\rangle_{\pM}$ denotes the inner product in $L^2(\pM)$.
\end{theorem}

This completeness result is also proved for any interior hypersurface $H$ in Theorem \ref{thm:interior} in \S \ref{H}.
The Cauchy data of $u_j$ on $H$ is defined as 
$$\begin{cases}
\text{Dirichlet data}: & \omega_j=u_j|_H,\\
\text{Neumann data}: & \psi_j=d_n u_j |_H,
\end{cases}$$
where $d_n$ is the normal derivative on $H$.

In \cite{BFSS}, the authors gave a justification for Theorem~\ref{thm:BFSS}, but it is not fully rigorous, as they used a Balian-Bloch approach involving sums of compositions of Green functions, but did not prove convergence of the infinite sums appearing in their formulae. In Section~\ref{sec:classical} we give an elementary, self-contained, and rigorous proof of this formula. In fact, we obtain the following improvement in which we identify the next term in the asymptotic expansion of the expression in \eqref{eq:BFSS}:
\begin{equation}\begin{gathered}
\frac{\pi}{2}\sum_j \frac{ \rho(\lambda-\lambda_j)\psi_j(y)\langle\psi_j,\phi\rangle }{\lambda_j^2} = \phi(y) - \frac1{2} \lambda^{-2} \left[\Delta_{\pM}-\frac14(n-1)^2H_y^2
+\frac12(n-1)(n-2)K_y\right] \phi(y)  \\  +  O(\lambda^{-3}),\quad\lambda\to\infty,
\end{gathered}\label{2ndterm}\end{equation}
where $\Delta_\pM$ is the induced (positive) Laplacian on $\pM$, and $H_y$ and $K_y$ are the mean and scalar curvatures of $\pM$ at $y$.

In addition, we give a straightforward modification of the proof to obtain the following analogue for Neumann boundary conditions:
\begin{proposition}[Completeness of boundary traces of Neumann eigenfunctions]\label{prop}
Let $\rho$ be as in Theorem~\ref{thm:BFSS}.  Then for any $\phi\in C^\infty(\pM)$, we have 
\begin{eqnarray*}
&&\frac{\pi}{2}\sum_j\rho(\lambda-\lambda_j)\langle\omega_j,\phi\rangle\omega_j(y)\\
&=&\phi(y) +\frac1{2} \lambda^{-2}\left[\Delta_{\pM}-\frac34(n-1)^2H_y^2+\frac12(n-1)(n-2)K_y\right]\phi(y)+  O(\lambda^{-3}),\quad\lambda\to\infty. 
\end{eqnarray*}
\end{proposition}

In Section~\ref{sec:semiclassical}, we prove a semiclassical generalization of the above results. To state this theorem, let $K^D_\lambda$ be the operator appearing in the RHS of \eqref{eq:BFSS}:
$$K^D_\lambda=\frac{\pi}{2}\sum_j\rho(\lambda-\lambda_j)\lambda_j^{-2}\psi_j\langle\psi_j,\cdot\rangle,$$
and
$K^N_\lambda$ denote the corresponding operator in Proposition~\ref{prop}:
$$K^N_\lambda=\frac{\pi}{2}\sum_j\rho(\lambda-\lambda_j)\omega_j\langle\omega_j,\cdot\rangle.$$ 

Then we show: 
\begin{theorem}\label{thm:semiclassical}
Let $A_h$ be a semiclassical pseudo-differential operator on $\pM$, microsupported in $\{(y,\eta)\in T^*(\pM):|\eta|<1-\varepsilon_1\}$ for some $\varepsilon_1>0$. Let $\rho$ be such that $\hat\rho$ is supported sufficiently close to $0$ (depending on $\varepsilon_1$). Then 
\begin{enumerate}[(i).]
\item $A_hK^D_{h^{-1}}$ and $K^D_{h^{-1}} A_h$ are semiclassical pseudo-differential operators with principal symbol
\begin{equation}
\sigma(A)(1-|\eta|^2)^{1/2};
\label{Dsymbol}\end{equation}
\item $A_hK^N_{h^{-1}}$ and $K^N_{h^{-1}} A_h$ are semiclassical pseudo-differential operators with principal symbol
\begin{equation}
\sigma(A)(1-|\eta|^2)^{-1/2}.
\label{Nsymbol}\end{equation}
\end{enumerate}
\end{theorem}

\begin{remark} This is closely related to \cite{BFS} which is a sequel to \cite{BFSS}. The main result, equations (22) and (23), of \cite{BFS} is as follows: let $\Omega \subset \RR^2$ be a two-dimensional domain with smooth boundary, and let $h_j(q, p)$ be the $j$th semiclassical Husimi function associated to $\psi_j$: that is, 
$$
h_n(q, p) = \frac1{2\pi \lambda_n} \big| \langle \psi_j, c_{q, p, \lambda} \rangle \big|^2,
$$
the square (up to normalization) of the inner product of $\psi_j$ with a coherent state $c_{q, p, \lambda}$ on $\partial \Omega$, centred at $(q, p)$:
$$
c_{q, p, \lambda}(s) = \big(\frac{\lambda}{\pi}\big)^{1/4} \sum_{m \in \ZZ} e^{i\lambda p(s-q+mL) - \lambda(s-q+mL)^2/2}, \quad L = |\partial \Omega |,
$$
where $s$ is arc length on $\partial \Omega$. Then the following asymptotic relation is derived (nonrigorously) in \cite{BFS}: 
\begin{equation}
 \rho(\lambda - \lambda_j) \sum_{\lambda_j \leq \lambda} h_n(q, p) = \frac{k}{\pi^2} \sqrt{1 - |p|^2} + O(k^{1/2}), \quad |p| < 1. 
\end{equation}
Since the symbol of a semiclassical pseudodifferential operator $A$ can be expressed as 
$$
\sigma(A)(q, p) = \big\langle A_h c_{q, p, h^{-1}}, c_{q, p, h^{-1}} \big\rangle 
$$
(see \cite[Example 1, Section 5.1]{Zworski}), we see that this follows from \eqref{Dsymbol}. 
\end{remark}

\begin{remark} In Section~\ref{sec:semiclassical} we show that Theorem~\ref{thm:semiclassical} implies Theorem~\ref{thm:BFSS}: see Remark~\ref{Thms}.
\end{remark} 

\subsection{Applications to Kuznecov sum formulae}

Theorem \ref{thm:BFSS} has an immediate application to Kuznecov sum formulae. The general Kuznecov formula is a singularity expansion for the distribution
\begin{equation} \label{St} S_H(t) = \int_H \int_H \cos t \sqrt{\Delta_B} (t, q, q') dS(q') dS(q), \end{equation}
where $H \subset M$ is a smooth submanifold,  $dS$ is a density
on $H$ and 
$B$ denotes either  Dirichlet or Neumann boundary conditions. In \cite{Z2} in the boundaryless case, the singularities of $S_H(t)$ are shown to correspond  to trajectories of the geodesic flow which intersect $H$ orthogonally at two distinct times, and to be singular at the difference $T$ of these times. We refer to such trajectories  as H-{\it  orthogonal} geodesics. A natural
problem is to generalize the Kuznecov formula to manifolds with boundary. In the boundary case, $H$ could be an interior
hypersurface, or the boundary $\partial M$, or a hypersurface which intersects the boundary in a variety of ways. 

We now partially generalize the   Kuznecov formula to  compact Riemannian manifolds $(M, g)$ with smooth boundary
to the case where $H = \partial M$.  

\begin{theorem}\label{K} 
Let $(M,g)$ be compact Riemannian manifold with smooth boundary $\partial M$. 
Let $\{\psi_j\}$ and $\{ \omega_j \}$ be the sequence of boundary traces of normalized Dirichlet, resp. Neumann, eigenfunctions on $M$.  Let $\phi\in C_0^{\infty}(\partial M)$. Then
$$\sum_{\lambda_j<\lambda}\lambda_j^{-2}\big|\langle\phi,\psi_j\rangle\big|^2=\frac{2}{\pi} \lambda\|\phi\|_{L^2(\pM)}^2 + O_\phi(1),$$
resp. 
$$\sum_{\lambda_j<\lambda}\big|\langle\phi,\omega_j\rangle\big|^2 = \frac{2}{\pi}\lambda\|\phi\|_{L^2(\pM)}^2+O_\phi(1).$$
\end{theorem}

Theorem \ref{K} follows from
 Theorem \ref{thm:BFSS}  by taking the inner product with $\phi$ on both sides of the equations therein.  In fact, Theorem \ref{thm:interior}
implies the partial  Kuznecov formula for interior hypersurfaces. The statement is similar to the above and is omitted.

The Kuznecov formula for $H = \partial M$ was used in \cite{JZ} to prove that the number of nodal domains on
non-positively curved surfaces with concave boundary tends to infinity along a density one subsequence of eigenvalues.
A self-contained proof was given in that special case, and moreover it was shown that  $S_{\partial M}(t)$ has an isolated,  conormal singularity at $t = 0$. We briefly sketch the proof in \S \ref{sec:K} for comparison.   Theorem \ref{K} should allow for further generalizations of
the nodal counting  results.  Note that the results only involve
the singularity expansion for some time interval $[0, \epsilon]$, and does not involve singularities corresponding to $\partial M$-orthogonal billiard trajectories for $t \not= 0$.

\subsection{Related mathematical literature}
Let us now discuss some related results in the mathematical literature. Theorem~\ref{thm:BFSS} suggests that the functions $\lambda_j^{-1} \psi_j$, in a suitably sized spectral window centred at $\lambda$, behave like an orthonormal basis as $\lambda \to \infty$. Recently some other results with a similar flavour have appeared. 
In \cite{B}, Barnett proved the following quasi-orthogonality result for the $\psi_j$:
\begin{theorem}[Pairwise quasi-orthogonality]
Let $\Omega \subset \RR^n$ be a Euclidean domain. Then there exists a constant $C$ depending only on $\Omega$ such that
$$\left|\int_{\pM}(\xx(y)\cdot\nn(y))\psi_i(y)\psi_j(y)dy-2\lambda_i^2\delta_{ij}\right|
\le C(\lambda_i^2 - \lambda_j^2)^2.$$
\end{theorem}

If $|\lambda_i - \lambda_j|$ is small then $(\lambda_i^2 - \lambda_j^2)^2$ is small compared to $\lambda_i^2$. Therefore this result says that the boundary traces of Dirichlet eigenfunctions in a small frequency window are close to orthogonal with respect to a weighted inner product (which is positive definite for starshaped domains) on $\pO$. 

Another result along these lines is the following recently proved by the second author and Barnett \cite{BH}:

\begin{theorem}[Spectral window quasi-orthogonality]\label{thm:BH}
Let $\Omega$ be as above and let $c > 0$. There exists a constant $C$ depending only on $c$ and $\Omega$ such that the operator norm bound 
$$\Big\|\sum_{|\lambda_j - \lambda |\le c}\psi_j\langle\psi_j,\cdot\rangle_\pO\Big\|_{L^2(\pO)\to L^2(\pO)}\le C\lambda^2$$
holds for all $\lambda \ge1$.
\end{theorem}

To understand the implication of this result, it is helpful to recall that there is a lower bound of $c \lambda_j^2$ for the square of the $L^2$ norm of each individual $\psi_j$ (or equivalently, the operator norm of the rank one operator $\psi_j \langle \psi_j, \cdot \rangle$). This result says that adding up $\sim \lambda^{n-1}$ of these rank one operators increases the operator norm by at most a constant factor, independent of $\lambda$. This is only possible if the $\psi_j$ are approximately orthogonal to each other (cf. the Cotlar-Stein Lemma). 

We next mention some semiclassical results about the distribution of boundary traces of eigenfunctions.  Let $A_h$ be a semiclassical pseudo acting on $L^2(\pO)$ as in Theorem~\ref{thm:semiclassical}. Then the \emph{local Weyl laws} for the $\psi_j$ and $\omega_j$ are as follows \cite{HaZe}:
\begin{equation}\begin{gathered}
\lim_{\lambda \to \infty} \frac1{N(\lambda)} \sum_{\lambda_j \leq \lambda} h_j^2 \langle A_{h_j} \psi_j, \psi_j \rangle = \frac{4}{\vol(S^{n-1}) \vol(\Omega)} \int_{B^* \pO} \sigma(A) (1 - |\eta|^2)^{1/2} \, dy d\eta, \\
\lim_{\lambda \to \infty} \frac1{N(\lambda)} \sum_{\lambda_j \leq \lambda}  \langle A_{h_j} \omega_j, \omega_j \rangle = \frac{4}{\vol(S^{n-1}) \vol(\Omega)} \int_{B^* \pO} \sigma(A) (1 - |\eta|^2)^{-1/2} \, dy d\eta.
\end{gathered}\label{LWL}\end{equation}

Here, $h_j = \lambda_j^{-1}$, $|\eta|$ is measured using the induced metric on $T^*(\pO)$, and $N(\lambda)$ is the eigenvalue counting function for $\Omega$, that is, the number of $\lambda_j$ (counted with multiplicity) less than or equal to $\lambda$. The statement can be interpreted as saying that the boundary traces are semiclassically concentrated inside the ball bundle of $\pO$ (that is, with semiclassical frequencies $\leq 1$), and are distributed as $(1 - |\eta|^2)^{\pm 1/4}$ with sign $+$ for Dirichlet and $-$ for Neumann. This is not surprising since the symbol of $h d_n$,  the semiclassical normal derivative operator, restricted to the characteristic variety of $h^2\Delta - 1$, is $
(1 - |\eta|^2)^{1/2}$ at the boundary, so we expect this discrepancy between the Dirichlet and Neumann distributions. This gives an explanation for the different powers of $(1 - |\eta|^2)$ in the Dirichlet and Neumann cases of Theorem~\ref{thm:semiclassical}. 

The local Weyl law tells us that boundary traces $\lambda_j^{-1} \psi_j$ are composed of frequencies up to $\lambda_j$, and the results of Barnett \cite{B} and Barnett-Hassell \cite{BH} say that they are approximately orthogonal. Theorem~\ref{thm:BFSS} adds to these heuristics a completeness statement: the boundary traces $\lambda_j^{-1} \psi_j$, for $\lambda_j$ in a frequency window of fixed, suitably chosen size centred at $\lambda$, behave like an orthonormal basis of the finite dimensional space of functions on $\pO$ with frequencies up to $\lambda$.


Using standard techniques, the statements in \eqref{LWL} could be modified so that the LHS involves averages of the quantities $h_j^2 \langle A_{h_j} \psi_j, \psi_j \rangle$ or $\langle A_{h_j} \omega_j, \omega_j \rangle$ over a frequency window of fixed width, or alternatively involving a frequency window function $\rho(\lambda - \lambda_j)$ as in Theorems~\ref{thm:BFSS} and \ref{thm:semiclassical}. This modified statement then follows by taking the trace of the operator $A_h K^D_{h^{-1}}$ or $A_h K^N_{h^{-1}}$, using Theorem~\ref{thm:semiclassical} together with the asymptotic
$$
\operatorname{tr} B \sim (2\pi h)^{-(n-1)} \int_{T^*(\partial M)} \sigma(B)(x, \xi) \, dx \, d\xi
$$
for the trace of a semiclassical pseudodifferential operator on $\partial M$ \cite[Chapter 9]{DS}.  (Actually there is a slight discrepancy in the two statements, as from Theorem~\ref{thm:semiclassical} we would get $A_h$ rather than $A_{h_j}$ in the inner product, but this is insignificant due to the rapid decay of the window function $\rho(\lambda - \lambda_j)$.) 
Thus, Theorem~\ref{thm:semiclassical} can be viewed as a refinement of the local Weyl law, in the sense that it is an operator statement whose trace gives a version of the the local Weyl law.

\

When the billiard flow is ergodic, we can strengthen \eqref{LWL} to quantum ergodicity for boundary traces. This is the statement that there is a density one subset $J$ of positive integers such that, restricting the index $j$ to $J$, we have 
\begin{equation}\begin{gathered}
\lim_{j \in J \to \infty} h_j^2 \langle A_{h_j} \psi_j, \psi_j \rangle = \frac{4}{\vol(S^{n-1}) \vol(\Omega)} \int_{B^* \pO} \sigma(A) (1 - |\eta|^2)^{1/2} \, dy d\eta, \\
\lim_{j \in J \to \infty} \langle A_{h_j} \omega_j, \omega_j \rangle = \frac{4}{\vol(S^{n-1}) \vol(\Omega)} \int_{B^* \pO} \sigma(A) (1 - |\eta|^2)^{-1/2} \, dy d\eta.
\end{gathered}\label{QE}\end{equation}

This result was proved by G\"erard-Leichtnam \cite{GL} in the Dirichlet case, then by the second and fourth authors for general boundary conditions and Euclidean domains \cite{HaZe}, and by Burq \cite{Bur} for Riemannian manifolds. 

\

Under special dynamical assumptions one can give more precise results on the frequency localization of
the boundary traces. We thus consider the Fourier coefficients of boundary traces
of eigenfunctions relative to   eigenfunctions of the boundary Laplacian $\Delta_{\partial \Omega}$. That is,
we expand $\omega_j$ (resp. $\psi_j$) 
\begin{equation}\label{EXPAND} 
\omega_j(y) = \sum_{k= 0 }^{\infty} \langle \omega_j, \phi_k \rangle_{L^2(\partial \Omega)}
\; \phi_k 
\end{equation} 
in terms of an orthonormal basis of  boundary eigenfunctions   $\phi_k$ of $\Delta_{\partial \Omega}$ with $\Delta_{\partial \Omega} \phi_k = \Lambda_k^2 \phi_k$, and consider
the size of the Fourier coefficients $\langle \omega_j, \phi_k \rangle_{L^2(\partial \Omega)}$. At least heuristically, the matrix 
\begin{equation}\label{U} 
U_{\lambda} : = \begin{pmatrix} \langle \omega_j, \phi_k \rangle_{L^2(\partial \Omega)} \end{pmatrix}_{\Lambda_k \leq \lambda, \lambda_j \in [\lambda - C(\lambda), \lambda]},
\end{equation} 
of Fourier coefficients is approximately unitary (where $C(\lambda)$ is chosen so that the matrix is square). 
Indeed, it is the change of basis matrix from the approximately orthonormal boundary traces to the orthonormal basis of boundary eigenfunctions with eigenvalues $\leq \lambda$. A natural question is the extent to which the entries deviate from randomness. It is obvious that in symmetric situations such as balls,
where can define joint eigenfunctions of the symmetry and the Laplacian for both the interior and boundary eigenfunctions, the Fourier coefficients will peak when the the eigenfunctions share the same symmetry and vanish otherwise. In general, to measure the size of the Fourier coefficients, we let $A_{h_j}$ in \eqref{QE}  be of the form form $\rho (h_j  \sqrt{\Delta_{\partial \Omega}})$ where $h_j = \lambda_j^{-1}$ and where  $\rho \in C_0^{\infty}(\R_+)$ is a smooth cutoff supported in $[0, 1]$ and equal to 1 on a smaller interval.

In the ergodic case, one would not expect any frequency localization and would expect \eqref{U} to be
similar to a random unitary matrix. To a large degree, the results above prove that. In the case of Neumann eigenfunctions on domains with ergodic billiards, we obtain, for a density one subsequence of $\lambda_j$, 
\begin{equation}\label{EXPANDB}  
\sum_{k: \Lambda_k \leq \lambda_j}^{\infty}| \langle \omega_j, \phi_k \rangle_{L^2(\partial \Omega)}
|^2  \; \rho\left(\frac{\Lambda_k}{\lambda_j}\right)   \to \frac{4}{\vol(S^{n-1}) \vol(\Omega)} \int_{B^* \pO}  \rho(|\eta|)  (1 - |\eta|^2)^{-1/2} \, dy d\eta. 
\end{equation} 

This shows that the squares $|\langle \omega_j, \phi_k \rangle_{L^2(\partial \Omega)}|^2$ of the Fourier coefficients with $\Lambda_k \leq \lambda_j$ are asymptotically of size 
$$
\Big(1 - \big(\frac{\Lambda_k}{\lambda_j}\big)^2\Big)^{-1/2} \frac{4}{\vol(S^{n-1}) \vol(\Omega)} \frac{\vol(B^* \partial \Omega)}{N_{\partial \Omega}(\lambda_j)} \sim \Big(1 - \big(\frac{\Lambda_k}{\lambda_j}\big)^2\Big)^{-1/2} \Big(\frac{2\pi}{\lambda_j}\Big)^{n-1} \frac{4}{\vol(S^{n-1}) \vol(\Omega)} 
$$
in an averaged sense. 
Moreover, the Fourier components are rapidly decaying for $\Lambda_k\gg\lambda_j$. The latter statement holds without any dynamical assumptions, but the size of the Fourier coefficients $|\langle \omega_j, \phi_k \rangle_{L^2(\partial \Omega)}|
^2$ will in general depend on the billiard dynamics and reflect the extent to which the frequencies of the boundary traces localize. 

In   \cite{TZ1, ctz},  the quantum ergodicity theorem for Cauchy data of \cite{HaZe,Bur,GL} along the boundary
is generalized to any hypersurface $H$  if  the billiard (geodesic) maps in $T^*(M)$ is ergodic.  For a general hypersurface, there are two components to the Cauchy data,
and quantum ergodicity refers to the pair. In \cite{TZ2}, the  Quantum Ergodic Restriction (QER)  is proved
for the individual Dirichlet and Neumann  if the hypersurface satisfies    an asymmetry condition with 
respect to the geodesic flow.  This condition is not needed in Theorem \ref{thm:interior}  to prove that Dirichlet and Neumann data are individually complete. The QER theorem is of a different nature than the completeness result since it
concerns individual eigenfunctions rather than averages over the spectrum, and its proof uses the long time behaviour of the wave kernel and not just the singularity at $t=0$.
  
There should  also exist  pointwise Weyl laws for boundary traces of eigenfunctions. They are stated in  \cite{Z2,TW}
but are not proved there.  We use the notation $u_j^b = \lambda_j^{-1} \psi_j$ for boundary traces of Dirichlet eigenfunctions and $u_j^b = \omega_j$ for boundary traces of Neumann eigenfunctions. Proposition 2.1 of \cite{Z2} states that there exist positive constants $C_D, C_N$ (depending only on the dimension) so that
$$\sum_{j: \lambda_j \leq \lambda}
|u_j^b(y)|^2  = \left\{\begin{array}{ll}  
C_D \lambda^{n } + O(\lambda^{n-1}), & \mbox{Dirichlet,}
\\ & \\
 C_N \lambda^n + O(\lambda^{n-1}), & \mbox{Neumann}.
\end{array} \right.$$
At least in the case of concave boundary, the proof in \cite{M} of the Weyl law for manifolds with concave boundary
should adapt to the boundary traces.

Moreover, the estimates of the remainder terms $R_{\lambda}(y) $ can be strengthened from $O$ symbols to $o$ symbols if  the set of loops with footpoint $y \in \partial \Omega$ has measure $0$
in $B^*_y\partial \Omega$. 
The  jump in the remainder $R_{\lambda}(y) $ is given by 
\begin{equation}\label{eig3} 
\sum_{j: \lambda_j = \lambda} |u_j^b(y)|
=\sqrt{R(\lambda,y)-R(\lambda-0^+,y)}.
\end{equation}

In \cite{SoZ} it is shown that
$$\sup_{y \in \partial \Omega} |u_j^b(y)| = o (\lambda^{(n-1)/2}), $$ 
if the set of billiard loops with footpoint at  $y \in \partial \Omega$ has measure $0$
in $B^*_y\partial \Omega$. 

\

\subsection*{Acknowledgements}  We thank Alex Barnett for helpful conversations. We also acknowledge the support of the Australian Research Council through a Future Fellowship FT0990895 (A.H.) and Discovery Project DP120102019 (A.H. and X.H.), and the National Science Foundation through DMS-1206527 (S.Z.) and DMS-1346706 (H.H).


\section{Proofs of Theorem~\ref{thm:BFSS} and Proposition~\ref{prop}}\label{sec:classical}

\subsection{Proof of Theorem~\ref{thm:BFSS}}
We start by defining $w\in C^\infty(M)$ to be the harmonic function with boundary value $\phi$. Then, using Green's formula, we have
\begin{eqnarray*}
\langle\psi_j,\phi\rangle_{\pM}&=&\langle d_nu_j,w\rangle_{\pM}-\langle  u_j, d_nw\rangle_{\pM}\\
&=&\langle-\Delta u_j,w\rangle_{M}+\langle u_j,\Delta w\rangle_M\\
&=&-\lambda_j^2\langle u_j,w\rangle_{M}.
\end{eqnarray*}

Therefore, 
$$\begin{gathered}
\frac{\pi}{2}\sum_j\rho(\lambda-\lambda_j)\lambda_j^{-2}\langle\psi_j,\phi\rangle\psi_j(x) 
= -\frac{\pi}{2}\sum_j \rho(\lambda-\lambda_j)\langle u_j,w\rangle_{M}\psi_j(x) \\ = -\frac{\pi}{2}\rho\ast\left(\sum_j\langle u_j,w \rangle_{M}\psi_j(x)\delta_{\lambda_j}\right)(\lambda).
\end{gathered}$$

We shall interpret the sum as follows: we take both the positive and negative square roots, so each eigenvalue $E_j$ gives rise to two terms above, one a multiple of $\delta_{\lambda_j}$ and the other the same multiple of $\delta_{-\lambda_j}$. Let $t$ be the dual variable to $\lambda$, then the Fourier transform of the RHS of the above equation is 
$$-\pi\hat\rho(t)\sum_j\langle u_j,w\rangle_{M}\psi_j(x)\cos(t\lambda_j)=-\pi\hat\rho(t)d_nv(t, x),$$
where $v(t,x)$ is the solution to the wave equation 
\begin{equation}
\begin{cases}
(\partial_t^2+\Delta_D)v(t,x)=0 & \text{ in }\R\times M,\\
v(0,x)=w(x) & \text{ if }x\in M,\\
\partial_tv(0,x)=0. &
\end{cases}
\label{wave-eqn}
\end{equation}

To prove the theorem, it is sufficient to show that
$$d_nv(t,x)=-2\phi(x)\cdot\delta(t)+f$$
for some $f\in L^1_{loc}$ in $t$. The remainder of the proof is devoted to showing this. 

Equation (2.1) is less innocent than it appears: the function $w$ is nonzero at the boundary, and therefore is not in the domain of the Dirichlet Laplacian $\Delta_D$. Nevertheless it is an $L^2$ function, and we can apply the solution operator $\cos t\sqrt{\Delta_D}$ to it. 

\begin{example}[An one-dimensional model case] 
Consider the Dirichlet wave equation
$$\begin{cases}
\partial^2_tv(t,x)+\Delta_Dv(t,x)=0 & \text{ in }\R\times[0,1],\\
v(0,x)=w(x)=1_{[0,1]}, &\\
\partial_tv(0,x)=0, &
\end{cases}$$
where $1_{[0,1]}$ is the characteristic function of the interval $[0,1]$. The formula of d'Alembert yields
$$v(t,x)=1_{[|t|,1-|t|]}(x)$$
for $|t|<1/2$. Therefore, 
\begin{enumerate}[(i).]
\item for $x$ near $0$, 
$$\partial_x v(t, x)= \delta(t-x)+\delta(t+x),$$
showing that 
$$d_n(t, 0) =-\partial_x v(t, 0)= -2 \delta(t);$$

\item for $x$ near $1$, 
$$\partial_x v(t, x)=-\delta(1-t-x)-\delta(1+t-x),$$
showing that 
$$d_n(t,1) =\partial_x v(t,1)= -2 \delta(t).$$
\end{enumerate}
\end{example}

\

One can also derive a similar result in $\R^2$ and $\R^3$ using spherical means Poisson's and Kirchhoff's formulae \cite[Section 2.4.1]{E}.

In the case of a domain with smooth boundary, it is convenient to work with  Fermi coordinates $(r,y)$ near the boundary, where $r$ is distance to the boundary, and $y$ are local coordinates on the boundary extended to a tubular neighbourhood in such a way that $y$ is constant on lines normal to the boundary when $r$ is small. In these coordinates the metric takes the form
$$g=dr^2+h_{ij}(r,y)dy^idy^j,$$
where $h_{ij}(0,y)$ is the induced metric on the boundary $\{r=0\}$, and the summation convention is in force. The Riemannian measure
$$dg=k^2drdy,$$
where $k^4=\det h_{ij}$. Write $u = kv$; then $u$ solves the equation 
\begin{equation}
\begin{cases}
(\partial_t^2+P)u(t,r,y)=0 & \text{ in }\R\times M,\\
u(0,r,y)=k(r,y)w(r,y) & \text{ if }(r,y)\in M,\\
\partial_tu(0,r,y)=0, &
\end{cases}
\end{equation}
where
$$P= k \Delta k^{-1} = -(\partial_r^2+\partial_{y_i}h^{ij}\partial_{y_j}+f),$$
in which $h^{ij}=(h_{ij})^{-1}$ and
$$f=-k^{-1}\partial_r^2k-k^{-1}\partial_{y_i}(h^{ij}\partial_{y_j}k)$$
is a smooth function in $M$.

Now we write down an approximate solution to equation (2.2). Notice that since $w$ is harmonic, or equivalently, $P(kw) = 0$,  the solution to the equation should be static for times $|t| < r$. On the other hand, motivated by the one-dimensional example, we expect to have a conormal singularity propagating out from the boundary. This can also be motivated by the idea that the initial data can be thought of as having a jump of magnitude $\phi$ at the boundary in order to satisfy the boundary condition. This is a conormal singularity that can be expected to propagate normal to the boundary for nonzero time. Thus we specify an ansatz for time $|t| < \epsilon$, $\epsilon$ small, of the form 
\begin{equation}
u_N(t,r, y)= \begin{cases}
H(r-t)(k w)(r, y) + H(t-r) \sum_{j=0}^N (t-r)^j b_j(r, y), \quad t>0, \\
H(r+t) (kw)(r, y) + H(-t-r) \sum_{j=0}^N (-t-r)^j b_j(r, y), \quad t < 0, 
\end{cases} 
\label{ansatz}\end{equation}
where $H$ is the Heaviside function. (We write this formula in Fermi coordinates near the boundary; it should be interpreted as meaning that $u(t, z) = w(z)$ whenever the distance from $z$ to the boundary is bigger than $|t|$.) 

We apply the wave operator $(\partial_t^2 + P)$ to \eqref{ansatz}, and obtain for $t > 0$,
\begin{equation}\begin{gathered}
(\partial_t^2 +P) u_N(t, r, y) = 
- 2 \delta(r-t)\partial_r (kw)(r, y)  + 2 \delta(r-t) (\partial_rb_0)(r, y)  \\ + H(t-r) \left[\sum_{j=0}^N (t-r)^j Pb_j(r,y)  
+ \sum_{j=1}^N2j (t-r)^{j-1} \partial_r b_j(r,y)  \right]. 
\end{gathered}\end{equation}

We therefore choose the $b_j$ to satisfy 
\begin{equation}\begin{gathered}
(\partial_rb_0)(r, y) = \partial_r (kw)(r, y), \\
(\partial_r b_{j})(r, y) = -\frac1{2j} Pb_{j-1} , \quad j \geq 1.
\end{gathered}
\end{equation}

To satisfy the Dirichlet boundary condition, we specify that $b_j(0, y) = 0$. This allows us to solve uniquely for the $b_j$: 
\begin{equation}\begin{gathered}
b_0(r, y) = (kw)(r, y) - (kw)(0,y) , \\
b_{j}(r, y) = -\frac1{2j} \int_0^r P b_{j-1}(s, y) \, ds , \quad j \geq 1.
\end{gathered}\label{bjrecurrence}
\end{equation}

With this choice of $b_j$ we find that 
\begin{equation}\begin{gathered}
(\partial_t^2 +P) u_N(t, r, y) = e_N(t, r, y),  \\
e_N(t, r, y)=  \begin{cases} H(t-r) (t-r)^N Pb_N, \quad t > 0, \\
H(-t-r) (-t-r)^N Pb_N, \quad t < 0 .
\end{cases}
\end{gathered}\label{errorterm}\end{equation}

In particular, the RHS is a $C^{N-1}$ function of $t$ with values in $L^2(M)$. We can solve this error term using Duhamel's formula \cite[Section 2.4.2]{E}: 
\begin{equation}
u(t, r, y) = u_N(t, r, y) -u_N'(t, r, y), \quad  u_N'(t, r, y) = \int_0^t \frac{\sin (t-s) \sqrt{\Delta_D}}{\sqrt{\Delta_D}} e_N(s, r,y) \, ds . 
\label{correction}\end{equation}

Notice that $e_N(t)$ is even in $t$, and hence so is $u_N'(t, r, y)$. Therefore, $u(t, r, y)$ is also even in $t$. We need the following information about the correction term $u_N'(t, r, y)$:
\begin{lemma}\label{lem:correction}
The term $u_N'(t, r, y)$ obeys the Dirichlet boundary condition, and  $\partial_r u_N'(t, r, y) |_{r = 0}$ is $C^{K}$ in time with values in $C^{K'}(\pM)$, if $N$ is sufficiently large relative to $K + K'$. 
\end{lemma}

\begin{proof}[Proof of Lemma~\ref{lem:correction}] 
We integrate by parts in the integral above, exploiting the fact that $u$ is differentiable in time, to get 
\begin{equation}
u'_N(t, r, y) =  \Delta_D^{-1} \int_0^t \left[\cos (t-s) \sqrt{\Delta_D} \right] \frac{d}{ds} e_N(s,r,y) \, ds + \Delta_D^{-1} e_N(t,r,y). 
\label{correction2}\end{equation}

Since $\Delta_D^{-1}$ maps $L^2(\pM)$ into the domain of $\Delta_D$, we see that the correction term obeys the Dirichlet boundary condition for all $t$. Moreover, we can iterate this procedure, obtaining an expression of the form 
\begin{eqnarray}
u'_N(t, r, y)&=&\Delta_D^{-1} e_N(t,r,y) - \Delta_D^{-2} \left(\frac{d}{dt}\right)^2 e_N(t, r, y) + \dots + (-1)^{k+1} \Delta_D^{-k} \left(\frac{d}{dt}\right)^{2k} e_N(t, r, y)\nonumber\\
&& + (-1)^{k+1} \Delta_D^{-k} \int_0^t  \left[\cos (t-s) \sqrt{\Delta_D} \right] \left(\frac{d}{dt}\right)^{2k} e_N(s, r,y) \, ds.\label{correction3} 
\end{eqnarray}
We use the standard mapping property that $\Delta_D^{-1}$ maps $H^k(M)$ to $H^{k+2}(M) \cap H^1_0(M)$ continuously, for all $k \geq 0$. (See, e.g. \cite[Sections 8.2, 8.3]{GT}.) 

Also, it is clear from \eqref{errorterm} that $e_N$ is a $C^{N-1-k}$ function of $t$ with values in $H^{k}(\pM)$. It follows from these facts and \eqref{correction3} that $u_N'$ is a $C^{N-1-2m}$ function of $t$ with values in $H^{2m}(M) \cap H^1_0(M)$, for $2m \leq N-1$. In particular, it obeys the Dirichlet boundary condition for all $t$. Moreover, taking the $r$-derivative and restricting to $r=0$ maps $H^{2m}(M)$ to $H^{2m-3/2}(\pM)$, and then to $C^k(\pM)$ provided that $2m - 3/2 > k + (n-1)/2$ by Sobolev embedding. Hence, the restriction of $\partial_r  u_N'(t, r, y) |_{r = 0}$ to $\pM$ is a $C^K$ function of $t$ with values in $C^{K'}(\pM)$ provided that $N > K + K' + (n+4)/2$. 
\end{proof}

We need to justify that the function $u$ just constructed really satisfies the \emph{Dirichlet} wave evolution for $|t| < \epsilon$. We begin by observing that $u(t)$ is clearly continuous in $t$ with values in $L^2(M)$. Because of this, it suffices to check that $u(t)$ satisfies the Dirichlet wave evolution for $t < 0$ and for $t > 0$, or equivalently, for $-\epsilon \leq t \leq -\epsilon'$ and $\epsilon' \leq t \leq \epsilon$ for arbitrary positive $\epsilon' < \epsilon$. It suffices to exhibit $u$ as a limit of Dirichlet wave solutions which lie in the domain of $\Delta_D$ for each $t$. This is easily done by smoothing out the singularity in $u(t)$, $|t| \in [\epsilon', \epsilon]$, without changing $u$ in a neighbourhood of the boundary. (Finite propagation speed ensures that we can do this on a whole time interval disjoint from a neighbourhood of $t=0$.) Therefore, $u$ indeed satisfies the Dirichlet wave evolution. 

We are interested in the limit of the  Fourier transform of $\hat \rho(t) d_n v(t, 0, y) = \hat \rho(t) d_n (k^{-1} u)(t,0,y)$ as $\lambda \to \infty$. 
Due to Lemma~\ref{lem:correction}, the contribution of the correction term $u_N'$ is $O(\lambda^{-K})$ provided $N$ is sufficiently large. Hence we only need to consider the $u_N$ term. 

Bearing in mind that the normal derivative is \emph{minus} the $r$-derivative:
$$-\pi \hat \rho(t) d_n (k^{-1} u)(t, 0, y)=\pi \hat \rho(t) \partial_r (k^{-1} u)(t, 0, y),$$
the terms
$$\begin{aligned} (&k^{-1}u_0)(t,r,y) \\ =
&\begin{cases}
H(r-t)w(r,y)+H(t-r)[w(r,y)-w(0,y)]=w(r,y)+H(t-r)w(0,y), & t>0\\
H(r+t)w(r,y)+H(-t-r)[w(r,y)-w(0,y)]=w(r,y)+H(-t-r)w(0,y), & t<0
\end{cases}\end{aligned}$$ 
contribute
$$\pi \hat \rho(t) \Big[ \delta(r-t) w(0, y) + \delta(r+t) w(0,y) + 2 \partial_r w(r, y) \Big].$$

Evaluating at $r=0$ we get 
$$\pi \hat \rho(t) \Big[ 2\delta(t)  \phi(y) + 2\partial_r w(0, y) \Big].$$

The contribution of the $b_1$ term to $-\pi \hat \rho(t) d_n (k^{-1} u)(t, 0, y)$ is
$$\pi \hat \rho(t) \Big[ H(t-r) (r-t) k^{-1}(0,y)  \partial_r b_1(r, y) + H(-t-r) (r+t) k^{-1}(0,y)\partial_r b_1(r, y) \Big]$$
since $b_1 = 0$ when $r=0$. Applying \eqref{bjrecurrence} and evaluating at $r=0$ shows that this equals 
\begin{eqnarray*}
&&-\pi \hat \rho(t)  \frac{|t|}{2}   k^{-1}(0,y) P b_{0}(0,y)\\
&=& \pi \frac{|t|}{2} \hat \rho(t) k^{-1}(0,y) \Big[ \partial_r^2 (kw)(0, y)  \Big]  \\
&=& \pi \frac{|t|}{2} \hat \rho(t) \Big(  \partial_r^2w + 2 \frac{\partial_rk}{k}\partial_r w + \frac{\partial_r^2k}{k} w \Big)(0,y) \\
&=& \pi \frac{|t|}{2} \hat \rho(t)  \left[\Delta_{\pM} + \frac{\partial_r^2k}{k}(0,y)\right]\phi(y).
\end{eqnarray*}

It is easy to check that the contribution of the other terms is a bounded function of $t$ and $y$ that is $O(t^2)$ near $t=0$. Taking the inverse Fourier transform gives 
$$\frac{\pi}{2}\sum_j\rho(\lambda-\lambda_j)\lambda_j^{-2}\psi_j(y)\langle\psi_j,\phi\rangle=\phi(y) - \frac1{2} \lambda^{-2} \left[\Delta_{\pM} + \frac{\partial_r^2k}{k}(0,y)\right] \phi(y)  +  O(\lambda^{-3}),\quad\lambda\to\infty.$$

Next, we use the first and second variations of area formula on the $(n-1)$-dimensional submanifold $\pM=\{r=0\}$ (See, e.g. \cite[Sections 8 and 9]{S}) to derive
$$\frac{\partial_r^2k}{k}(0,y)=\frac14(n-1)^2H_y^2-\frac12Tr(\II_y^2),$$
where $H$ is the mean curvature, and $\II_y$ is the second fundamental form on $\pM$. From \cite[Section 4.4]{dC}, we have
$$Tr(\II_y^2)=(n-1)^2H_y^2-(n-1)(n-2)K_y,$$
therefore,
\begin{eqnarray*}
&&\frac{\pi}{2}\sum_j\rho(\lambda-\lambda_j)\lambda_j^{-2}\psi_j(y)\langle\psi_j,\phi\rangle\\
&=&\phi(y) - \frac1{2} \lambda^{-2} \left[\Delta_{\pM}-\frac14(n-1)^2H_y^2+\frac12(n-1)(n-2)K_y\right] \phi(y)  +  O(\lambda^{-3}),\quad\lambda\to\infty,
\end{eqnarray*}
showing \eqref{2ndterm}. In particular, Theorem~\ref{thm:BFSS} is proved. 

\begin{example}[Dirichlet eigenfunction expansion in the unit disc]
We investigate the above expansion in the unit disc $\{(r,\theta)\in\R^2:0\le r\le1,0\le\theta\le2\pi\}$, that is, $\Delta_\pM=-\partial_\theta^2$ and
\begin{equation}\label{disc}
\frac{\pi}{2}\sum_j\rho(\lambda-\lambda_j)\lambda_j^{-2}\psi_j(\theta)\langle\psi_j,\phi\rangle=\phi(\theta) - \frac1{2} \lambda^{-2} \left(-\partial_\theta^2-\frac14\right) \phi(\theta)  +  O(\lambda^{-3}),\quad\lambda\to\infty.
\end{equation}

However, we have the Dirichlet eigenfunctions in this case as
$$u_{k,l}(r,\theta)=c_{k,l}J_k(\lambda_{k,l}r)e^{ik\theta}.$$

Here, $J_k$ is the Bessel function of the first kind and order $k$, $\lambda_{k,l}$ is the $l$-th zero of $J_k$, and $c_{k,l}$ is the normalisation factor \cite[Section 2.63]{F}:
$$c_{k,l}=\frac{1}{\sqrt\pi J_{k+1}(\lambda_{k,l})}=\frac{1}{\sqrt\pi J_{k}'(\lambda_{k,l})},$$
from \cite[Section 9.5.4]{AS}. Therefore, 
$$\psi_{k,l}(\theta)=d_nu_{k,l}(r,\theta)=\partial_ru_{k,l}(r,\theta)=c_{k,l}\lambda_{k,l}J'_k(\lambda_{k,l})e^{ik\theta}=\frac{\lambda_{k,l}}{\sqrt\pi}\,e^{ik\theta},$$
and
\begin{eqnarray}
&&\frac{\pi}{2}\sum_{k,l}\rho(\lambda-\lambda_{k,l})\lambda_{k,l}^{-2}\psi_{k,l}(\theta)\langle\psi_{k,l},\phi\rangle\nonumber\\
&=&\pi\sum_{k,l}\rho(\lambda-\lambda_{k,l})\cdot\check\phi(k)e^{ik\theta}\nonumber\\
&=&\pi\sum_k\left[\sum_l\rho(\lambda-\lambda_{k,l})\right]\cdot\check\phi(k)e^{ik\theta}\label{poisson},
\end{eqnarray}
using Poisson's summation formula \cite[Section 7.2]{H1}:
$$\pi\sum_l\rho(\lambda-\lambda_{k,l})=1-\frac18\lambda^{-2}\left(4k^2-1\right)+O(\lambda^{-3}),\quad\lambda\to\infty.$$

Here, we used the fact from \cite[Section 9.5.12]{AS} that for fixed $k$ and $l\gg k$,
$$\lambda_{k,l}=\beta-\frac{4k^2-1}{8\beta}+O(\beta^{-3}),$$
where $\beta=(l+\frac12k-\frac14)\pi$. Thus,
\begin{eqnarray*}
\eqref{poisson}&=&\sum_k\check\phi(k)e^{ik\theta}-\frac12\lambda^{-2}\sum_k\left(k^2-\frac14\right)\check\phi(k)e^{ik\theta}+O(\lambda^{-3})\sum_k\check\phi(k)e^{ik\theta}\\
&=&\phi(\theta) - \frac1{2} \lambda^{-2} \left(-\partial_\theta^2-\frac14\right) \phi(\theta)  +  O(\lambda^{-3}),\quad\lambda\to\infty,
\end{eqnarray*}
and therefore we have recovered \eqref{disc}.

\end{example}

\

Using Theorem~\ref{thm:BFSS} together with Theorem~\ref{thm:BH}, we obtain 
\begin{corollary}\label{cor}
Let $\rho$ be as above. Then the operators
$$K^D_\lambda=\frac{\pi}{2}\sum_j\rho(\lambda-\lambda_j)\lambda_j^{-2}\psi_j\langle\psi_j,\cdot\rangle$$
converge strongly to the identity operator in $B(L^2(\pM))$. 
\end{corollary}

\begin{proof}
Theorem~\ref{thm:BH} shows that the operators $K^D_\lambda$ are uniformly bounded as $\lambda\to\infty$. Therefore, it is only necessary to show that $K^D_\lambda\phi\to\phi$ in $L^2(\pM)$ for a dense subset. This is shown above for $\phi\in C^\infty(\pM)$, so we are done. 
\end{proof}

\subsection{Proof of Proposition~\ref{prop}}
Since this runs parallel to the proof of Theorem~\ref{thm:BFSS}, we provide only a sketch. 

Let $\phi \in C^\infty(\pM)$, without loss of generality we may assume that  
$$\int_{\pM} \phi(y) d\sigma(y) = 0,$$
otherwise we only need to replace $\phi$ by $\phi-\frac{1}{\vol(\pM)}\int_\pM\phi$. We define $w$ to be the unique harmonic function that is orthogonal to constants and such that $d_n w = \phi$ at $\pM$. 
We then have
$$\langle\omega_j,\phi\rangle_{\pM}=\langle v_j,d_nw\rangle_{\pM}-\langle d_nv_j,w\rangle_{\pM}=\langle v_j,-\Delta w\rangle_M-\langle-\Delta v_j,w\rangle_{M}=\langle\Delta v_j,w\rangle_{M}$$

Therefore, 
$$\begin{gathered}
\frac{\pi}{2}\sum_j\rho(\lambda-\lambda_j)\langle \omega_j ,\phi\rangle\omega_j(x) 
= \frac{\pi}{2}\sum_j \rho(\lambda-\lambda_j)\langle \Delta v_j,w\rangle_{M}\omega_j(x) \\ 
=\frac{\pi}{2}\rho\ast\left(\sum_j\langle \Delta v_j,w \rangle_{M}\omega_j(x)\delta_{\lambda_j}\right)(\lambda).
\end{gathered}$$

The Fourier transform of the RHS of the above equation is 
$$\pi\hat\rho(t)\sum_j\langle \Delta v_j,w\rangle_{M}\omega_j(x)\cos(t\lambda_j)=\pi\hat\rho(t) \Delta v(t, x) \Big|_{\pM},$$
where $v(t,x)$ is the solution to the wave equation 
\begin{equation}
\begin{cases}
(\partial_t^2+\Delta_N)v(t,x)=0 & \text{ in }\R\times M,\\
v(0,x)=w(x) & \text{ if }x\in M,\\
\partial_tv(0,x)=0. &
\end{cases}
\label{wave-eqn-N}\end{equation}

We change, as above, to $u = kv$; then $u$ solves the equation 
$$
\begin{cases}
(\partial_t^2+P)u(t,r,y)=0 & \text{ in }\R\times M,\\
u(0,r,y)=k(r,y)w(r,y) & \text{ if }(r,y)\in M,\\
\partial_tu(0,r,y)=0. &
\end{cases}
$$

As before, we have an initial condition that doesn't satisfy the boundary condition. We write down an ansatz for the solution. One difference is that, for the Neumann Laplacian, we expect the leading singularity in the solution to be a jump in the derivative of the function, rather than a jump in the solution itself. Therefore, our ansatz takes the form 
\begin{equation}
u_N(t,r, y)= \begin{cases}
k(r, y) w(r, y) + H(t-r) \sum_{j=1}^N (t-r)^j b_j(r, y), \quad t > 0, \\
k(r,y) w(r, y) + H(-t-r) \sum_{j=1}^N (-t-r)^j b_j(r, y), \quad t < 0 , 
\end{cases} 
\label{ansatzN}\end{equation}
with the sum starting from $j=1$ rather than $j=0$ in the Dirichlet case. 
This gives rise to equations of the form 
$$\begin{gathered}
\partial_r b_1 = 0, \\
\partial_r b_j = - \frac1{2j} P b_{j-1}, \quad j \geq 2.
\end{gathered}$$

Imposing the Neumann boundary condition on $v = k^{-1} u$  gives 
$$\begin{gathered}
k^{-1} b_1(0, y) =  \partial_r w(0, y) ,\\
k^{-1} b_j(0, y) =  \frac{1}{j} \partial_r (k^{-1}b_{j-1})(0, y) \implies b_j(0,y) = \frac{1}{j}  \left( \partial_r b_{j-1} - \frac{\partial_rk}{k} b_{j-1} \right)(0,y) , \quad j \geq 2.
\end{gathered}$$

This gives a unique solution for these functions: 
\begin{equation}\begin{gathered}
b_1(r, y) = k(0,y)(\partial_r w)(0, y)  , \\
b_{j}(r, y) = \frac{1}{j}  \left( \partial_r b_{j-1} - \frac{\partial_rk}{k} b_{j-1} \right)(0,y) - 
\frac{1} {2j}  \int_0^r P b_{j-1}(s, y) \, ds , \quad j \geq 2.
\end{gathered}\label{bjrecurrenceN}
\end{equation}

Now we compute $\pi\hat\rho(t)\Delta v=\pi\hat\rho(t)k^{-1}(Pu)$,
\begin{eqnarray*}
&&\pi\hat\rho(t)k^{-1}(r,y)P(u_N)(t,r,y)\\
&=&\pi\hat\rho(t)k^{-1}(r,y)\big[-\delta(t-r)b_1(r,y) - \delta(-t-r)b_1(r,y)\big]\\
&&-\pi\hat\rho(t)k^{-1}(r,y)H(t-r)\Big[\sum_{j=2}^Nj(j-1)(t-r)^{j-2}b_j(r,y)\Big]\\
&&-\pi\hat\rho(t)k^{-1}(r,y)H(-t-r)\Big[\sum_{j=2}^Nj(j-1)(-t-r)^{j-2}b_j(r,y)\Big].
\end{eqnarray*}

Taking $N=3$ and evaluating at $r=0$, the above equation equals
\begin{eqnarray*}
&&\pi\hat\rho(t)\big[-\delta(t)\partial_rw(0,y)-\delta(-t)\partial_rw(0,y)\big]-2\pi\hat\rho(t)k^{-1}(0,y)b_2(0,y)-6\pi\hat\rho(t)k^{-1}(0,y)|t|b_3(0,y)\\
&=&2\pi\hat\rho(t)\delta(t)\phi(y)-\pi\hat\rho(t)k^{-1}(0,y)\partial_rk(0,y)\phi(y)-\frac{\pi}{2}\hat \rho(t)|t|\left[\Delta_{\pM} + \frac{k\partial_r^2k-2(\partial_rk)^2}{k^2}(0,y)\right]\phi(y),
\end{eqnarray*}
noticing that $\partial_rw(0,y)=-d_nw(0,y)=-\phi(y)$. Taking the inverse Fourier transform gives 
$$\frac{\pi}{2}\sum_j\rho(\lambda-\lambda_j)\omega_j(y)\langle\omega_j,\phi\rangle=\phi(y)+\frac1{2} \lambda^{-2}\left[\Delta_{\pM}+\frac{k\partial_r^2k-2(\partial_rk)^2}{k^2}(0,y)\right]\phi(y)  +  O(\lambda^{-3}),\quad\lambda\to\infty.$$

Similarly as in \S2.1, we have
$$\frac{k\partial_r^2k-2(\partial_rk)^2}{k^2}(0,y)=-\frac34(n-1)^2H_y^2+\frac12(n-1)(n-2)K_y.$$

\begin{example}[Neumann eigenfunction expansion in the unit disc]
We also investigate the above expansion in the unit disc as in \S 2.1,
\begin{equation}\label{discN}
\frac{\pi}{2}\sum_j\rho(\lambda-\lambda_j)\omega_j(\theta)\langle\omega_j,\phi\rangle=\phi(\theta)+\frac1{2} \lambda^{-2} \left(-\partial_\theta^2-\frac34\right) \phi(\theta)  +  O(\lambda^{-3}),\quad\lambda\to\infty.
\end{equation}

However, we have the Neumann eigenfunctions in this case as
$$v_{k,l}(r,\theta)=c_{k,l}J_k(\lambda_{k,l}'r)e^{ik\theta},$$
where $\lambda_{k,l}'$ is the $l$-th zero of $J_k'$, and $c_{k,l}$ is the normalisation factor:
$$c_{k,l}=\frac{\lambda_{k,l}'}{J_k(\lambda_{k,l}')\sqrt{\pi(\lambda_{k,l}'^2-k^2)}},$$
from \cite[Section 11.4.2]{AS}. Therefore, 
$$\omega_{k,l}(\theta)=v_{k,l}(1,\theta)=\frac{\lambda_{k,l}'}{\sqrt{\pi(\lambda_{k,l}'^2-k^2)}}e^{ik\theta},$$
and
\begin{equation}\label{poissonN}
\frac{\pi}{2}\sum_{k,l}\rho(\lambda-\lambda_{k,l}')\omega_{k,l}(\theta)\langle\omega_{k,l},\phi\rangle
=\pi\sum_k\left[\sum_l\frac{\lambda_{k,l}'^2}{\lambda_{k,l}'^2-k^2}\cdot\rho(\lambda-\lambda_{k,l}')\right]\cdot\check\phi(k)e^{ik\theta},
\end{equation}
a similar computation as in \S2.1 shows that
$$\pi\sum_l\rho(\lambda-\lambda_{k,l}')=1-\frac18\lambda^{-2}\left(4k^2+3\right)+O(\lambda^{-3}),\quad\lambda\to\infty.$$

Here, we used the fact from \cite[Section 9.5.13]{AS} that for fixed $k$ and $l\gg k$,
$$\lambda_{k,l}'=\beta-\frac{4k^2+3}{8\beta}+O(\beta^{-3}),$$
where $\beta=(l+\frac12k-\frac34)\pi$. Thus,
\begin{eqnarray*}
&&\pi\sum_l\frac{\lambda_{k,l}'^2}{\lambda_{k,l}'^2-k^2}\cdot\rho(\lambda-\lambda_{k,l}')\\
&=&\pi\sum_l\rho(\lambda-\lambda_{k,l}')+\pi\sum_l\frac{k^2}{\lambda_{k,l}'^2-k^2}\cdot\rho(\lambda-\lambda_{k,l}')\\
&=&\pi\sum_l\rho(\lambda-\lambda_{k,l}')+\pi\frac{k^2}{\lambda^2}\sum_l\rho(\lambda-\lambda_{k,l}')+\pi k^2\sum_l\frac{\lambda_{k,l}'^2-\lambda^2}{\lambda^2(\lambda_{k,l}'^2-k^2)}\cdot\rho(\lambda-\lambda_{k,l}')\\
&=&1+\frac18\lambda^{-2}\left(4k^2-3\right)+O(\lambda^{-3}),\quad\lambda\to\infty,
\end{eqnarray*}
and therefore we have recovered \eqref{discN} if we plug this into \eqref{poissonN}.

\end{example}

\begin{remark}
Corollary~\ref{cor} does \emph{not} hold for $K_\lambda^N$. In fact, the uniform boundedness principle implies that if the $K_\lambda^N$ converge strongly, then they are uniformly bounded in operator norm. But this is not true on the unit disc for example, where the norm of the $\omega_j$ can be as large as $c \lambda_j^{1/3}$. 
\end{remark}


\section{Proof of Theorem~\ref{thm:semiclassical}}\label{sec:semiclassical}

\subsection{Boundary traces of wave kernels}
In this section, we prove Theorem~\ref{thm:semiclassical}. We continue to use Fermi coordinates $x=(r,y)$ near $\pM$, with dual coordinates $\xi=(\xi_n,\eta)$. We let $\tilde g=h_{ij}(0,y)$ denote the induced metric on the boundary, as in \S 2.1.

We let $R_y$ and $R_{y'}$ denote restriction operators to $\pM$ the left, resp. right, factor.   Consider the operators $K^D_\lambda$ and $K^N_\lambda$ defined in the introduction, taking the Fourier transform gives us the operators 
$$\pi\hat \rho(t) R_yR_{y'}d_{n_y}d_{n_{y'}} \frac{\cos(t\sqrt{\Delta_D})}{\Delta_D}$$
in the Dirichlet case, and 
$$\pi\hat \rho(t) R_yR_{y'} \cos(t\sqrt{\Delta_N})$$
in the Neumann case. In \cite{HZ} the operators 
$$R_yR_{y'}d_{n_y}d_{n_{y'}} \frac{\sin(t\sqrt{\Delta_D})}{\sqrt{\Delta_D}}\text{ and }
R_yR_{y'} \frac{\sin(t\sqrt{\Delta_N})}{\sqrt{\Delta_N}}$$
were analyzed in both the Dirichlet and Neumann case, and it is straightforward to 
adapt their results to obtain the following lemma. For brevity, we will call either of the operators above the ``boundary trace of the wave kernel''. 

\begin{lemma}\label{HZ-result}
Suppose that $\hat \rho$ is supported in $[-\epsilon, \epsilon]$ and equal to $1$ in a neighbourhood of $0$. Let $\chi(y, D_t, D_y)$ be a pseudodifferential operator on $\RR \times \pM$ with symbol of the form 
\begin{equation}
\chi(y, \tau, \eta) = \zeta( |\eta|_{\tilde g}^2/\tau^2) (1 - \phi(\eta, \tau)), 
\label{zeta}\end{equation}
where $\zeta(s)$ is supported where $s \leq 1 - \delta$ for some positive $\delta$, and $\phi \in C_c^\infty(\RR^n)$ is equal to 1 near the origin. Then, for sufficiently small $\epsilon$ (depending on $\delta$), 
\begin{enumerate}[(i).]
\item the kernels of $$\hat \rho(t) \chi(y, D_t, D_y) \circ R_yR_{y'}d_{n_y}d_{n_{y'}} \frac{\cos(t\sqrt{\Delta_D})}{\Delta_D}, \quad\hat \rho(t)  R_yR_{y'}d_{n_y}d_{n_{y'}} \frac{\cos(t\sqrt{\Delta_D})}{\Delta_D} \circ \chi(y, D_t, D_y)$$
are  distributions conormal to $\{ y = y' , t = 0 \}$ with principal symbol 
\begin{equation}
2\chi(y, \tau, \eta) \left( 1 - \frac{|\eta|_{\tilde g}^2}{\tau^2} \right)^{\frac12};
\label{wavesymbolD}\end{equation}

\item the kernels of 
$$\hat \rho(t) \chi(y, D_t, D_y) \circ R_yR_{y'}\cos(t\sqrt{\Delta_N}), \quad\hat \rho(t) R_yR_{y'}\cos(t\sqrt{\Delta_N}) \circ \chi(y, D_t, D_y)$$ 
are distributions conormal to $\{ y = y' , t = 0 \}$ with principal symbol 
\begin{equation}
2\chi(y, \tau, \eta) \left( 1 - \frac{|\eta|_{\tilde g}^2}{\tau^2} \right)^{-\frac12} . 
\label{wavesymbolN}\end{equation}
\end{enumerate}
\end{lemma}

\begin{proof}[Proof of Lemma~\ref{HZ-result}]
The proof is essentially contained in \cite[Proposition 4]{HZ} (see Remark~\ref{manifolds}), so we only provide brief remarks here
about the minor differences between what is claimed in Lemma~\ref{HZ-result} and the results of \cite{HZ}. 

We first explain why $\epsilon$ has to be sufficiently small. It is well known that the boundary trace of the wave kernel has wavefront set contained in the set 
$$\begin{gathered}
\big\{ (t, \tau, y, \eta, y', -\eta') \mid \tau \neq 0, \  \text{ there exists a generalized bicharacteristic } \gamma \text{ in } T^*(\overline M)\\
\text{ of length } t \text{ such that } \gamma(0) \in T^*_y(\overline M), \ \gamma(t) \in T^*_{y'}(\overline M), \ 
\pi(\gamma(0)) = (y, \eta/\tau), \ \pi(\gamma(t)) = (y', \eta'/\tau) \big\}.
\end{gathered}$$

Here $\pi$ is the projection from $T^*_y(\overline M)$ to $T^*_y(\pM)$. 
Now suppose we use a cutoff function $\chi$ on the right of the boundary trace of the wave kernel. Then this removes all wavefront set with $|\eta'/\tau|_{y'} \geq 1 - \delta$. In particular, it removes all covectors generating bicharacteristics (geodesics) that are nearly tangent to the boundary. Since the boundary of $M$ is compact and smooth by hypothesis, this means that there is a positive time $\epsilon$, uniform over $y' \in \pM$, such that no bicharacteristic with initial condition $\eta'/\tau$ with length $\leq 1 - \delta$ reaches the boundary in time $\leq \epsilon$. It follows that composing with $\hat \rho(t) \chi$ on the right removes all wavefront set except that at $t=0$. 
But at $t=0$, points in the wavefront satisfy $y = y'$ and $\eta = -\eta'$, so this removes all the nearly tangential points in the left variables $(y, \eta)$ as well. This means that it is unnecessary to have a cutoff pseudo on the left. Similarly, if we have a cutoff pseudo on the left, and $\epsilon$ is chosen as above relative to $\delta$, then we do not need a cutoff pseudo on the right. 

It follows from \cite[Proposition 4]{HZ} that the kernels in (i) and (ii) in the lemma are conormal to $\{ y = y', t = 0 \}$. The precise symbols that we want in the lemma are not calculated, but it is straightforward to deduce \eqref{wavesymbolD} and \eqref{wavesymbolN} from \cite{HZ}. In the Dirichlet case, the symbol of the operator 
$$\hat \rho(t) \chi(y, D_t, D_y) \circ R_yR_{y'}d_{n_y}d_{n_{y'}} \frac{\sin(t\sqrt{\Delta_D})}{\sqrt{\Delta_D}}$$
was computed to be
$$ C\tau \chi(y, \tau, \eta) \left( 1 - \frac{|\eta|_{\tilde g}^2}{\tau^2} \right)^{\frac12},$$
but the constant $C$ was not calculated explicitly. However one can compute (see the remark following this proof) that the correct constant is $C=-2i$ and hence the symbol is
\begin{equation}
-2 i\tau \chi(y, \tau, \eta) \left( 1 - \frac{|\eta|_{\tilde g}^2}{\tau^2} \right)^{\frac12}.
\label{D-symbol}\end{equation} 

Notice that this operator is the $t$-derivative of the operator in part (i) of the lemma, up to smoothing terms (when the derivative hits the $\hat \rho$ factor, the result is a smoothing operator). We also notice that applying a $t$-derivative to a distribution conormal to $t=0, y = y'$ brings down a factor $i\tau$  matching the symbol of the operator in \cite{HZ}. It follows by the fundamental theorem of calculus that the difference between the kernel in (i) and \eqref{wavesymbolD} is constant in time. But 
due to the absence of wavefront set for $t \neq 0$, the difference has no wavefront set, i.e.~ is a smooth kernel, which is (in a trivial sense) also conormal to 
$y = y', t = 0$ (of order $-\infty$). This proves \eqref{wavesymbolD}.

In the Neumann case, the symbol of the operator 
$$
\hat \rho(t) \chi(y, D_t, D_y) \circ R_yR_{y'}\frac{\sin(t\sqrt{\Delta_N})}{\sqrt{\Delta_N}}
$$
can be computed similarly; we obtain 
\begin{equation}
2(i\tau)^{-1} \chi(y, \tau, \eta) \left( 1 - \frac{|\eta|_{\tilde g}^2}{\tau^2} \right)^{-\frac12}.\label{N-symbol}\end{equation}

By differentiating in $t$, we obtain the kernel in (ii) above, and this brings down a factor of $i\tau$ to give the principal symbol claimed in the lemma. 
\end{proof}

\begin{remark}
Let us give a sketch of the calculation of the constant in (3.4). We know that microlocally away from the tangential directions, the operator $\frac{\sin(t\sqrt{\Delta_D})}{\sqrt{\Delta_D}}$ is a Fourier integral operator and its wavefront relation is given by
$$\bigcup_{j\in \mathbb Z} W^j_{\pm}.$$ 
Here
$$W^j_{\pm} =\{(t,\tau, x, \xi, x', -\xi') \in T^* (\mathbb R \times M \times M) |\;\; \Phi^t(x',\xi')=(x,\xi), \tau=\pm|\xi|\;\text{and \textit{property}}\;j\; \text{is satisfied} \},$$ where `\textit{property} $j$' means, for $j > 0$, that $t>0$ and on the interval $[0,t]$, the orbit $\Phi^s(x',\xi')$ of the billiard flow reflects at $\partial M$ exactly $j$ times; similarly, for 
$j < 0$, that $t<0$ and on the interval $[t,0]$, the orbit $\Phi^s(x',\xi')$ of the billiard flow reflects at $\partial M$ exactly $|j|$ times. The relation with $j=0$ is just the diagonal relation, with $t = 0$, $x=x', \xi = \xi'$. In  \cite{HZ}, $\Gamma^j_{\pm}$ is used to denote the corresponding canonical relation i.e. $(W^j_\pm)'$. 

It is known that the symbol of $\frac{\sin(t\sqrt{\Delta_D})}{\sqrt{\Delta_D}}$ on $W_\pm^j$ is
$ \frac{(-1)^j}{2i\tau} \sigma,$ where $\sigma=|dt \wedge dx \wedge d\xi|^{1/2}$ is the canonical graph half-density (see \cite{HZ}). We would like to take normal derivatives, restrict to $\partial M$, and compute the symbol of the composition. To do this we use Fermi normal coordinates $(y,r)$ along $\partial M$, that is, $x=\text{exp}_y(r \nu_y)$ where $\nu_y$ is the interior unit normal at $y$ . Let $\xi = (\eta, \xi_n) \in T_{(y,r)}^*\R^n$ denote the corresponding symplectically
dual fiber coordinates. Taking normal derivatives in $r$ and $r'$ directions simply multiplies the symbol $\frac{(-1)^j}{2i\tau} \sigma$ by $i\xi_n$ and $-i\xi_n'$. Before we restrict our symbol to $T^* (\R) \times T^*(\partial M \times \partial M)$, we first restrict $W_\pm^j$ to $T^* (\mathbb R) \times T_{\partial M \times \partial M}^*(M \times M)$.  One can see that because we are away from the tangential directions and because $t$ is small, after this restriction we get a singularity only at $t=0$ and only when $j=0, 1$ and $-1$. In fact in the support of $\hat \rho (t) \chi$, the restriction of $W_\pm^0$ is  
$$A_{\pm}=\{(0, \tau, y, \xi, y', -\xi') | \;y=y', \xi=\xi', \tau=\pm |\xi|\} ,$$
the restriction of $W_\pm^1$ is
$$B^1_\pm=\{(0, \tau, y, \xi, y', -\xi') |\; y=y', \bar \xi=\xi', \tau=\pm |\xi|, \xi_n>0\} ,$$
and the restriction of $W_\pm^{-1}$ is
$$B_\pm^{-1}=\{(0, \tau, y, \xi, y', -\xi') | \;y=y', \bar \xi=\xi', \tau=\pm |\xi|, \xi_n<0\} ,$$
where $\bar\xi =(\eta, -\xi_n)$. For $|j|>1$, the restriction of $W^j_\pm$ is the empty set.  We note that the images of $A_\pm$ and $B_\pm^{\pm1}$ under the projection map 
$\pi: T^* (\mathbb R) \times T_{\partial M \times \partial M}^*(M \times M) \to T^* (\R) \times T^*(\partial M \times \partial M)$ are identical. However $\pi$ is a fold map on $A_\pm$ and injective on $B_\pm^{\pm 1}$.  

Using \cite[Equation 31]{HZ}, the restriction to 
$$\pi(A_\pm)=\pi(B^1_\pm)=\pi(B_\pm^{-1})=\{(0, \tau, y, \eta, y', -\eta') \in T^*(\R \times \partial M \times \partial M) |\;\; y=y', \eta=\eta'\}$$ 
of the half-density $\sigma= |dt \wedge dx \wedge d\xi|^{1/2}$  is given by
$$ \Big(1-\frac{|\eta|^2}{\tau^2}\Big)^{-1/2}|d\tau \wedge dy \wedge d\eta|^{1/2}.$$ 

This is basically because $\tau=\pm \sqrt{|\eta|^2+\xi_n^2}$ and therefore 
$d\xi=d\eta \wedge d \xi_n=\pm\frac{\tau}{\xi_n} d\eta \wedge d \tau$. 
Since $\pi$ is a fold map on $A_\pm$, we count the symbol on $A_\pm$ twice but we count it once for $B_\pm^{1}$ and $B_{\pm}^{-1}$. Hence the symbol of $$\hat \rho(t) \chi(y, D_t, D_y) \circ R_yR_{y'}d_{n_y}d_{n_{y'}} \frac{\sin(t\sqrt{\Delta_D})}{\sqrt{\Delta_D}}$$ is

$$\Big(2(i\xi_n)(-i\xi_n)(\frac{1}{2i\tau})+ (i\xi_n)(i\xi_n)(-\frac{1}{2i\tau})+(i\xi_n)(i\xi_n)(-\frac{1}{2i\tau})\Big)\Big(1-\frac{|\eta|^2}{\tau^2}\Big)^{-1/2} \chi(y, \tau, \eta)  |d\tau \wedge dy \wedge d\eta|^{1/2},$$
which simplifies to \eqref{D-symbol}. 

The following example of the half space also confirms the constants in \eqref{wavesymbolD} and \eqref{wavesymbolN}. Furthermore, it gives an illustration of the $0,1$ and $-1$ reflection terms.
\end{remark}

\begin{example}
Consider the operator $$ R_yR_{y'}d_{n_y}d_{n_{y'}} \Delta_D^{-1}\cos(t\sqrt{\Delta_D})$$ for a half space in $\R^n$. The kernel of $\Delta_D^{-1}\cos(t\sqrt{\Delta_D})$ on $\RR^{n}$ is given by
$$(2\pi)^{-n} \int e^{i(x-y) \cdot \xi} |\xi|^{-2} \cos t|\xi| \, d\xi.$$

Therefore, the kernel of $\Delta_D^{-1}\cos(t\sqrt{\Delta_D})$ on the half space $\RR^{n}_+$ where $x_n \geq 0$ is 
$$(2\pi)^{-n} \int \Big( e^{i(x-y) \cdot \xi} - e^{i(x-\overline{y}) \cdot \xi} \Big) |\xi|^{-2} \cos t|\xi| \, d\xi,$$
where $\overline{y} = (y_1, y_2, \dots, -y_n)$. Taking the derivative in $x_n$ and $y_n$ and then setting $x_n = y_n = 0$, we obtain with $x' = (x_1, \dots, x_{n-1})$,
$$(2\pi)^{-n} \int \int e^{i(x'-y') \cdot \xi'}  \frac{2 \xi_n^2}{|\xi|^{2}}  \cos t|\xi| \, d\xi' \, d\xi_n. $$

We localize in phase space away from tangential directions by multiplying by a cutoff $\zeta(|\xi'|/|\xi|)$, where $\zeta(s)$ is supported where $s \leq 1 - \delta$. This gives us 
$$(2\pi)^{-n} \int \int e^{i(x'-y') \cdot \xi'}  \frac{2 \xi_n^2}{|\xi|^{2}} \zeta\left( \frac{|\xi'|}{|\tau|}\right)  \cos t|\xi| \, d\xi' \, d\xi_n. $$

Since this is even in $\xi_n$ we can restrict the region of integration to $\xi_n \geq 0$ and double the integrand. Also expanding $\cos t |\xi|$, we obtain 
$$(2\pi)^{-n} \int \int_0^\infty  e^{i(x'-y') \cdot \xi'}  \frac{2 \xi_n^2}{|\xi|^{2}}  \zeta\left( \frac{|\xi'|}{|\tau|}\right) \Big( e^{it|\xi|} + e^{-it|\xi|} \Big)  \, d\xi' \, d\xi_n. $$

Now we change variable to $\tau = |\xi| = \sqrt{ |\xi'|^2 + \xi_n^2} \geq 0$. Then $d\xi' d\xi_n = \tau d\tau d\xi'/\xi_n$. So we can write 
\begin{eqnarray*}
&&(2\pi)^{-n} \int \int_0^\infty  e^{i(x'-y') \cdot \xi'}  \frac{2 \xi_n}{\tau} \zeta\left( \frac{|\xi'|}{|\tau|}\right) \Big( e^{it\tau} + e^{-it\tau} \Big)  \, d\xi' \, d\tau\\
&=& (2\pi)^{-n} \int \int_0^\infty  e^{i(x'-y') \cdot \xi'} 2 \sqrt{1 - |\xi'|^2/\tau^2} \  \zeta\left( \frac{|\xi'|}{|\tau|}\right) \Big( e^{it\tau} + e^{-it\tau} \Big)  \, d\xi' \, d\tau.
\end{eqnarray*}

We can change this into an integral in $\tau$ from $-\infty$ to $+\infty$:
$$(2\pi)^{-n} \int \int_{-\infty}^\infty  e^{i(x'-y') \cdot \xi'} 2 \sqrt{1 - |\xi'|^2/\tau^2} \ \zeta\left( \frac{|\xi'|}{|\tau|}\right) e^{it\tau}   \, d\xi' \, d\tau.$$

This shows that the symbol of $ R_yR_{y'}d_{n_y}d_{n_{y'}} \Delta_D^{-1}\cos(t\sqrt{\Delta_D})$ is $2\sqrt{1 - |\xi'|^2/\tau^2}$ in the region $|\xi'| < |\tau|$, confirming \eqref{wavesymbolD}. 
\end{example}

\begin{remark}\label{manifolds}
We note that \cite{HZ} is written only for Euclidean domains. In the present setting,  the cutoff $\chi$ removes nearly tangential geodesics, and the cutoff $\rho(t)$ means that we only consider propagation for small times. Together these cutoffs remove the difficulties caused by tangential propagation and multiple reflection from the boundary. 
In the presence of these cutoffs, the computation in \cite{HZ} extends to the case of Riemannian manifolds with smooth boundary.

\end{remark}

\subsection{Proof of Theorem~\ref{thm:semiclassical}}

We write the proof only for the Dirichlet case, as the Neumann case it is essentially identical. 

Let $A_h$ be a semiclassical pseudo as in the statement of Theorem~\ref{thm:semiclassical}, and consider the composition $K^D_{h^{-1}} A_h$. By assumption, the symbol $a(y, \eta)$ vanishes where $|\eta|_{\tilde g} \geq 1 - \varepsilon_1$. 
We choose a cutoff pseudo $\chi(y, D_t, D_y)$ as above, such that $\zeta(s)$ in \eqref{zeta} is equal to 1 for $s \leq 1 - \varepsilon_1/2$, and $0$ for $s \geq 1 - \varepsilon_1/4$. We write 
$$K^D_{h^{-1}} = K^D_{h^{-1}, \chi} +  K^D_{h^{-1}, 1 - \chi},$$
where the Fourier transform of $K^D_{h^{-1}, \chi}$ is
$$\pi\hat\rho(t)  R_yR_{y'}d_{n_y}d_{n_{y'}} \frac{\cos(t\sqrt{\Delta_D})}{\Delta_D} \circ \chi(y, D_t, D_y),$$
and where  the Fourier transform of $K^D_{h^{-1}, 1-\chi}$ is 
$$\pi\hat\rho(t)  R_yR_{y'}d_{n_y}d_{n_{y'}} \frac{\cos(t\sqrt{\Delta_D})}{\Delta_D} \circ \Big(\Id - \chi(y, D_t, D_y) \Big). $$

Correspondingly, we write 
\begin{equation}
K^D_{h^{-1}} A_h = K^D_{h^{-1}, \chi} A_h + K^D_{h^{-1}, 1- \chi}  A_h.
\label{KA}\end{equation}

We claim that the second term on the RHS of \eqref{KA} is a smooth kernel with all derivatives $O(h^\infty)$. To see this, we write out the composition as an integral. Writing the Fourier transform of $K^D_{h^{-1}}$ as $S$, the composition $K^D_{h^{-1}, 1-\chi}$ is given by
$$\begin{gathered}
\frac{1}{2\pi}\int e^{it/h} S(t, y, y'') e^{i[ (y''-y''')\cdot \eta + (t - t')\tau] } \big[1-\zeta(|\eta|_{\tilde g}^2/\tau^2)(1 - \phi(\eta, \tau) )\big] \\ \times e^{i(y''' - y') \cdot \eta'/h} a(y', \eta', h) dt' \, d\eta \, d\eta' \,  d\tau \, dy'' \, dy'''\,dt.
\end{gathered}$$

Making a semiclassical scaling in the $\eta, \tau$ variables, i.e., $\etabar = h \eta, \taubar = h \tau$, we can write this 
$$\begin{gathered}
\frac{1}{2\pi h^n} \int e^{it/h} S(t, y, y'') e^{i[(y''-y''')\cdot \etabar + (t- t')\taubar]/h } \left[1-\zeta\left(\frac{|\etabar|_{\tilde g}^2}{\taubar^2}\right)\left(1 - \phi\left(\frac{\bar\eta}{h}, \frac{\taubar}{h}\right)\right)\right] \\ \times e^{i(y''' - y') \cdot \eta'/h} a(y', \eta', h) dt' \, d\etabar \, d\eta' \,  d\taubar \, dy'' \, dy'''\,dt.
\end{gathered}$$

The phase is stationary only when $\taubar = 1$ and $\etabar = \eta'$. However, we see that the integrand vanishes in a neighbourhood of this set, due to the vanishing properties of $a$, $\zeta$ and $\phi$. It follows that we can integrate by parts, using the identity 
$$ih \frac{ (\etabar - \eta') \cdot \partial_{y'''} + (\taubar - 1) \cdot \partial_t}{|\etabar - \eta'|^2 + (\taubar - 1)^2} e^{i \Phi/h} = e^{i\Phi/h}, \quad \Phi =t + (y''-y''')\cdot \etabar + (t' - t)\taubar + (y''' - y') \cdot \eta' .$$

(Notice that this differential operator does not affect the $S(t, y, y'')$ kernel at all, nor the factor $1 - \phi$.) Integrating by parts $N$ times gives a factor of $h^N$, showing that this integral is $O(h^\infty)$. Derivatives are treated in the same way. 

Thus, the second term in \eqref{KA} is a trivial semiclassical pseudo. Now consider the first term in \eqref{KA}. By Lemma~\ref{HZ-result}, $S(t, y, y') \circ \chi$ is conormal to $\{ y=y', t = 0\}$ with principal symbol $2\pi\zeta(|\eta|_{\tilde g}^2/\tau^2)  \big( 1 - |\eta|_{\tilde g}^2/\tau^2 \big)^{1/2}$. That is, it can be written 
\begin{equation}\begin{gathered}
\frac{1}{(2\pi)^n}\int e^{i[(y - y') \cdot \eta + t \tau]} \left[ 2\pi\zeta \left( \frac{|\eta|_{\tilde g}^2}{\tau^2} \right) \left( 1 - \frac{|\eta|_{\tilde g}^2}{\tau^2} \right)^{\frac12} \big(1 - \phi(\eta, \tau)\big) + r(t, y, \eta, \tau)  \right]\, d\eta \, d\tau + R(t, y, y'),
\end{gathered}\label{www}\end{equation}
where $r$ is a symbol of order $-1$ and $R$ is smooth. If we take the inverse Fourier transform, then the $R$ term gives us something $O(\lambda^{-\infty})$ in $C^\infty(\pM \times \pM)$, which composes with $A_h$ to give a semiclassical operator of order $-\infty$. The composition of the integral in \eqref{www} with $A_h$ is an expression of the form
\begin{equation*}\begin{gathered}
\frac{1}{(2\pi h)^{n-1}\cdot(2\pi)^n} \int e^{i[(y - y'') \cdot \eta + t \tau]} \left[ 2\pi\zeta \left( \frac{|\eta|_{\tilde g}^2}{\tau^2} \right) \left( 1 - \frac{|\eta|_{\tilde g}^2}{\tau^2} \right)^{\frac12} \big(1 - \phi(\eta, \tau)\big) + r(t, y, \eta, \tau)  \right] \\
\times e^{i (y'' - y') \cdot \eta' /h}    a(y', \eta', h) 
\, d\eta  \, d\eta' \, dy'' \, d\tau .
\end{gathered}\end{equation*}

Changing variables to $\etabar$ and $\taubar$ as before, and taking the inverse Fourier transform,
$$\begin{gathered}
K_{h^{-1},\chi}^DA_h=\frac{1}{(2\pi h)^{n-1}\cdot (2\pi h)^n\cdot2\pi} \int e^{it/h} e^{i[(y - y'') \cdot \etabar + t \taubar]/h} e^{i (y'' - y') \cdot \eta' /h}  \\
\times  \left[ 2\pi\zeta \left( \frac{|\etabar|_{\tilde g}^2}{\taubar^2} \right) \left( 1 - \frac{|\etabar|_{\tilde g}^2}{\taubar^2} \right)^{\frac12} \left( 1- \phi\left(\frac{\bar\eta}{h}, \frac{\taubar}{h}\right)\right) + r\left(t, y, \frac\etabar h, \frac\taubar h\right)  \right]   a(y', \eta', h) 
\, d\etabar  \, d\eta' \, dy'' \, dt \, d\taubar .
\end{gathered}$$

Notice that the phase in this integral is nonstationary in $t$ for $\taubar \neq -1$. In particular, if we localize to $\{ |\taubar| \leq 1/2 \}$ using a smooth cutoff function, then the integral is $O(h^\infty)$ as follows by integrating by parts repeatedly in $t$. Therefore we can insert a cutoff function $\tilde \phi(\taubar)$ supported in $\{ |\taubar |\geq 1/4 \}$. On the support of $\tilde \phi$, the factor $1 - \phi$ is identically $1$ (for small $h$) and so we can remove this cutoff. Thus, up to $O(h^\infty)$ errors, the integral above is equivalent to 
\begin{equation}\begin{gathered}
\frac{1}{(2\pi h)^{n-1}\cdot (2\pi h)^n\cdot2\pi} \int e^{it/h} e^{i[(y - y'') \cdot \etabar + t \taubar]/h} e^{i (y'' - y') \cdot \eta' /h}  \\
\times  \left[ 2\pi\zeta \left( \frac{|\etabar|_{\tilde g}^2}{\taubar^2} \right) \left( 1 - \frac{|\etabar|_{\tilde g}^2}{\taubar^2} \right)^{\frac12}  + r\left(t, y, \frac\etabar h, \frac\taubar h\right)  \right]  \tilde \phi(\taubar) a(y', \eta', h) 
\, d\etabar  \, d\eta' \, dy'' \, dt \, d\taubar .
\end{gathered}\label{Kchi-int}\end{equation}

The phase function in \eqref{Kchi-int} can be written $(y - y') \cdot \eta' + \Phi$, where 
$$\Phi(t,\taubar,y'',\etabar)=t(1+\taubar)+(y-y'')(\etabar-\eta')$$
has a nondegenerate critical point at $t = 0$, $\taubar =-1$, $y'' = y$, and $\etabar = \eta'$.  We can perform stationary phase in the $(t, \taubar, y'', \etabar)$ variables, and we get an expression of the form 
$$\frac{1}{(2\pi h)^{n-1}}\int e^{i (y - y') \cdot \eta' /h}  b(y', \eta', h) \, d\eta',$$
where $b(y', \eta', h)$ is given by 
\begin{eqnarray*}
&=&\frac{1}{(2\pi h)^n}\int e^{i\Phi/h}\left[ \zeta \left( \frac{|\etabar|_{\tilde g}^2}{\taubar^2} \right) \left( 1 - \frac{|\etabar|_{\tilde g}^2}{\taubar^2} \right)^{\frac12} + r\left(t, y, \frac\etabar h,\frac\taubar h\right)\right] \tilde \phi(\taubar) a(y',\eta',h)\,d\etabar\,dy'' \, dt \, d\taubar\\
&=&a(y',\eta',h)(1-|\eta'|_{\tilde g}^2)^{1/2}+ O(h)
\end{eqnarray*}
since we chose $\chi$ such that $\zeta = 1$ on the support of $a$. This completes the proof of Theorem~\ref{thm:semiclassical} in the Dirichlet case.

\begin{remark}\label{Thms}[Theorem~\ref{thm:semiclassical} implies Theorem~\ref{thm:BFSS}] To see this, take $A_h$ to be a semiclassical operator that is has symbol $1$ near zero frequency and supported away from $|\eta|=1$. Then $A_h  \phi - \phi$ will be $O(h^\infty)$ in $L^2(\pM)$. So we can write 
$$K_{h^{-1}} \phi = K_{h^{-1}} A_h \phi + K_{h^{-1}} (A_h \phi - \phi). $$

The second term is $O(h^\infty)$ in $L^2(\pM)$. On the other hand, by Theorem~\ref{thm:semiclassical}, the first term is a pseudodifferential operator with symbol equal to $1$ at zero frequency. Any such operator converges strongly to the identity operator as $h \to 0$, showing convergence of $K_{h^{-1}} \phi$ to $\phi$ in $L^2(\pM)$ as in Corollary~\ref{cor}.  Pointwise convergence may be shown by commuting derivatives through $K_{h^{-1}} A_h$ and using Sobolev embedding theorems. 
\end{remark}

\section{\label{H} Completeness of Cauchy data on interior hypersurfaces}

In this section, we investigate the corresponding theorems of completeness of Cauchy data on interior hypersurfaces: Given a compact and smooth manifold $n$-dimensional $(M,g)$, let $\{u_j\}_{j=1}^{\infty}$ be an orthonormal basis of eigenfunctions of the positive Laplacian $\Delta$ with eigenvalues
$$0<\lambda_1^2 < \lambda_2^2 \le \lambda_3^2 \le\cdots.$$

Here, if $M$ has boundary, we impose the standard Dirichlet or Neumann or Robin boundary condition, which guarantees that $\Delta$ is a positive self-adjoint operator with discrete spectrum. We only assume the boundary is piecewise smooth. 

$H\Subset M$ is a smooth $(n-1)$-dimensional orientable hypersurface, we define the Cauchy data of $u_j$ on $H$ as
$$\begin{cases}
\text{Dirichlet data}: & \omega_j=u_j|_H,\\
\text{Neumann data}: & \psi_j=d_n u_j,
\end{cases}$$
where $d_n$ is the normal derivative on $H$. Then we establish the following theorem. 

\begin{theorem}[Completeness of Cauchy data on interior hypersurfaces]\label{thm:interior}
Let $\rho\in\mathcal S(\R)$ be as in Theorem \ref{thm:BFSS}. Then for any $\phi\in C^\infty(H)$, we have 
\begin{equation}\phi(x)=\lim_{\lambda\to\infty}\pi\sum_j\rho(\lambda-\lambda_j)\langle\omega_j,\phi\rangle\omega_j(x)
\label{int-completeness-D}\end{equation}
and
\begin{equation}\phi(x)=\lim_{\lambda\to\infty}\pi\sum_j\rho(\lambda-\lambda_j)\lambda_j^{-2}\langle\psi_j,\phi\rangle\psi_j(x),
\label{int-completeness-N}\end{equation}
where $\langle\cdot,\cdot\rangle=\langle\cdot,\cdot\rangle_{\pM}$ denotes the inner product in $L^2(H)$.
\end{theorem}

\begin{remark}\hfill
\begin{enumerate}
\item A result analogous to \eqref{int-completeness-D} can be proved for any submanifold with dimension $0\le k\le n-1$. In this case, there will be a power $\lambda^{n-1-k}$ and a constant depending on $k$ on the RHS. In the case of a point, $k=0$, this result goes back to  H\"ormander \cite{Hor}. 
\item Note that the boundary traces of Dirichlet and Neumann eigenfunctions are the Neumann and Dirichlet data on the boundary.
\item Comparing with the completeness identities in Theorem \ref{thm:BFSS} and Proposition \ref{prop}, the constant in the identities of Theorem \ref{thm:interior} is $\pi$ instead of $\pi/2$. This is roughly because the summation in terms of $\{\omega_j\}$ or $\{\psi_j\}$ only contains ``half'' of the Cauchy data on the interior hypersurfaces, while in the boundary case, the summation involves the ``whole'' Cauchy data (since the other half vanishes). 
\end{enumerate}
\end{remark}

Similar to the boundary case, Theorem \ref{thm:interior} is a consequence of the following semiclassical theorem.

\begin{theorem}\label{thm:interior semiclassical}
Let $A_h$ be a semiclassical pseudo-differential operator on $H$, microsupported in $\{(y,\eta)\in T^*(H):|\eta|<1-\varepsilon_1\}$ for some $\varepsilon_1>0$. Let $\rho$ be such that $\hat\rho$ is supported sufficiently close to $0$ (depending on $\varepsilon_1$). Then 
\begin{enumerate}[(i).]
\item $A_hC^D_{h^{-1}}$ and $C^D_{h^{-1}} A_h$ are semiclassical pseudo-differential operators with symbol
$$\sigma(A)(1-|\eta|^2)^{-1/2};$$
\item $A_hC^N_{h^{-1}}$ and $C^N_{h^{-1}} A_h$ are semiclassical pseudo-differential operators with symbol
$$\sigma(A)(1-|\eta|^2)^{1/2},$$
\end{enumerate}
where
$$C^D_\lambda=\pi\sum_j\rho(\lambda-\lambda_j)\omega_j\langle\omega_j,\cdot\rangle,\quad\text{and}\quad C^N_\lambda=\pi\sum_j\rho(\lambda-\lambda_j)\lambda_j^{-2}\psi_j\langle\psi_j,\cdot\rangle.$$ 
\end{theorem}

The proof of Theorem \ref{thm:interior semiclassical}, and therefore Theorem \ref{thm:interior}, is similar to Section \ref{sec:semiclassical}, the key ingredient is to study the canonical relations and principal symbols of Fourier transform of the operators $C^D_\lambda$ and $C^N_\lambda$ in $\lambda$ ($t$ is the dual variable of $\lambda$ as before):
$$2\pi\hat\rho(t)R_yR_{y'}\cos(t\sqrt\Delta)\quad\text{and}\quad2\pi\hat\rho(t)R_yR_{y'}d_{n_y}d_{n_{y'}}\frac{\cos(t\sqrt\Delta)}{\Delta}$$
away from the tangential directions. We omit the details here and only point out the main difference with the proof of Lemma \ref{HZ-result}: Unlike the boundary case, there is no reflection of billiards on the hypersurface, and one can compose the classical FIOs after proper microlocal cutoff. By using the same cutoff function $\chi$ as in Lemma \ref{HZ-result}, we have 
\begin{enumerate}[(i).]
\item the kernels of 
$$\hat\rho(t)\chi(y,D_t,D_y)\circ R_yR_{y'}\cos(t\sqrt\Delta),\quad\hat\rho(t)R_yR_{y'}\cos(t\sqrt\Delta)\circ\chi(y, D_t,D_y)$$ 
are distributions conormal to $\{y=y',t=0\}$ with principal symbol 
$$\chi(y, \tau, \eta) \left( 1 - \frac{|\eta|_{\tilde g}^2}{\tau^2} \right)^{-\frac12};$$

\item the kernels of 
$$\hat\rho(t)\chi(y,D_t,D_y)\circ R_yR_{y'}d_{n_y}d_{n_{y'}}\frac{\cos(t\sqrt\Delta)}{\Delta},\quad\hat\rho(t)  R_yR_{y'}d_{n_y}d_{n_{y'}}\frac{\cos(t\sqrt\Delta)}{\Delta}\circ\chi(y, D_t, D_y)$$
are distributions conormal to $\{ y = y' , t = 0 \}$ with principal symbol 
$$\chi(y,\tau,\eta)\left(1-\frac{|\eta|_{\tilde g}^2}{\tau^2}\right)^{\frac12},$$
\end{enumerate}
in which $\tilde g$ is the induced metric on $H$. The constants in the principal symbols differ with the ones in Lemma \ref{HZ-result} by 2, and this is because interior geodesics all pass through $H$, while in Lemma \ref{HZ-result}, the $0,-1,1$ reflections contribute. Then the rest of the proof of Theorem \ref{thm:interior semiclassical} is identical with the argument in Section \ref{sec:semiclassical}.

\begin{remark} In the torus $[0,a]\times[0,b]$, the Dirichlet and Neumann eigenfunctions are simple sine and cosine functions, and therefore one can write the Cauchy data on interior hypersurfaces $H=\{x=constant\}$ or $H=\{y=constant\}$. With the help of Poisson summation formula, one can compute the expansion by Dirichlet or Neumann data on $H$ as an example of Theorem \ref{thm:interior}.
\end{remark}

\section{\label{sec:K} Kuznecov sum formula: Sketch of proof of Theorem \ref{K}}

In this section, we  sketch the proof of Theorem \ref{K} in the case $H = \partial M$  by the method of \cite{JZ} rather than by using
Theorem \ref{thm:BFSS}.  The comparison between the proofs may illuminate the additional issues involved
in proving the pointwise result in Theorem \ref{thm:BFSS} rather than the weak convergence result of Theorem \ref{K}.
In fact, we show that  $S_{\partial M}(t)$ \eqref{St} has an isolated,  conormal singularity at $t = 0$. We let $dS$
denote the standard surface area form on $\partial M$ and let $f \in C^{\infty}(\partial M)$. 
In the Dirichlet or  Neumann case, we consider
\begin{equation}  \begin{array}{lll} S_f(t): & = & \int_{\partial M} \int_{\partial M}  E_B^b(t, q, q') f(q) f(q') dS(q) dS(q')
\\ &&\\ & = &  \sum_{j} \cos t \sqrt{\lambda_j} \left|\int_{\partial M} f(q)  \phi_j(q) dS(q) \right|^2 .
\end{array} \end{equation}
We then introduce a smooth cutoff  $\rho \in \scal(\R)$ with $\mbox{supp} \hat{\rho} \subset (-\epsilon, \epsilon)$,
where $\hat{\rho}$ is the Fourier transform of $\rho$, and consider
$$S_f(\lambda, \rho) = \int_{\R} \hat{\rho} (t) \;S_f(t) e^{i t \lambda} dt. $$

\begin{proposition} \label{CONOR} If supp $\hat{\rho}$ is contained in a sufficiently small interval around $0$, with
$\hat{\rho} \equiv 1$ in a smaller interval, 
$S_f(\lambda, \rho)$ is a semi-classical Lagrangian distribution whose asymptotic expansion in the Dirichlet case 
is given by 

\begin{equation}
S_f(\lambda, \rho) =  \frac{\pi}{2}\sum_j\rho(\lambda-\lambda_j)\lambda_j^{-2} |\langle\psi_j,\phi\rangle|^2
= ||\phi||_{L^2(\partial M)}^2 + o(1), 
\label{eq:BFSSb}\end{equation}  and in the Neumann case by 

\begin{equation}
S_f(\lambda, \rho) = \frac{\pi}{2}\sum_j\rho(\lambda-\lambda_j) |\langle\omega_j,\phi\rangle|^2
=  ||\phi||_{L^2(\partial M)}^2 + o(1).\end{equation}  
\end{proposition}

\begin{proof}

 There exists $\epsilon_0 > 0$ so that the 
\begin{equation} \label{ep0} \mbox{sing supp} S_f(t) \cap (- \epsilon_0, \epsilon_0) = \{0\}. \end{equation}

This follows from propagation of singularities for the wave kernel and its restriction
to the boundary. It is known that $WF(E_B(t, x, y))$ on a smooth domain consists 
of broken geodesic  trajectories,  which may in part  glide along the boundary. The pullback
to the boundary $E_B^b(t, q, q')$ forces the broken trajectories contributing to $WF(E^b)$  to begin and end on $\partial M$ and
integration over $\partial M$ forces them to be orthogonal to the boundary at both endpoints. Hence there exists
$\epsilon_0 > 0$ so that no trajectory starting orthogonally from a regular point of $\partial M$ can hit
$\partial M$ again at any point. Thus the only singularity in this time interval is $t = 0$.

 For $\epsilon < \epsilon_0$,  we only need to determine the
contribution of the main singularity of $S_f(t)$  at $t = 0$. 
 As in \cite{Z2} (1.6) we express $S_f(t)$ and $S_f(\lambda, \rho)$ in terms of pushforward under the submersion
$$\pi: \mathbb{R} \times \partial M \times \partial M \to \mathbb{R}, \;\;\; \pi(t, q, q') = t. $$
By Lemma \ref{ep0},  for $t \in (-\epsilon, \epsilon) $,
\begin{equation} \begin{array}{lll}
WF&(S_f^{\epsilon}(t)  ) =
&\{(0, \tau): \pi^* (0, \tau) = (0, \tau, 0, 0) \in WF  \cos t \sqrt{\Delta_B} (t, q, q') \} .
\end{array}\end{equation}
These wave front elements correspond to the points $(0, \tau, \tau  \nu_q, \tau  \nu_q) \in T^*_0 \mathbb{R}
\times T^*_{q, in} M \times T^*_{q, in} M$, i.e. where both covectors are co-normal to $\partial M$. Indeed,
as in (1.6) of \cite{Z2} the wave front set of $S_f(t)$ is the set
$$\{(t, \tau) \in T^* \mathbb{R}: \exists (x, \xi, y, \eta) \in C'_t \cap N^*(\partial \Omega) \times N^* \partial \Omega\} $$
in the support of the symbol.
 Thus, we may neglect the tangential part of the wave kernel in  determining the asymptotics of $S_f(\rho, \lambda)$
and microlocalize to the normal directions. The  non-tangential part of the wave kernel (in the normal directions
to $\partial M$)  has a geometric optics Fourier
integral representation, i.e. $S_f(t)$ is classical co-normal at $t = 0$. The remainder of the proof is similar
to that of Lemma \ref{HZ-result}, and is therefore omitted.

\end{proof}

\end{document}